\documentclass[12pt]{article}
\usepackage{graphicx} 

\usepackage[margin=1.2in]{geometry}

\usepackage{amssymb,amsmath,amsthm, enumerate, xcolor, amsfonts, hyperref}
\usepackage{geometry}

\newtheorem{theorem}{Theorem}[section]
\newtheorem{lemma}[theorem]{Lemma}
\newtheorem{corollary}[theorem]{Corollary}
\newtheorem{definition}[theorem]{Definition}

\newtheorem{observation}[theorem]{Observation}
\newtheorem{conjecture}[theorem]{Conjecture}
\newtheorem{claim}[theorem]{Claim}
\newtheorem{remark}[theorem]{Remark}

\title{The extremal cases of the Erd\H os--S\'os conjecture}
\author{Bruce Reed\footnote{Institute of Mathematics, Academia Sinica, Taiwan; \texttt{bruce.al.reed@gmail.com} Supported by  NSTC Grant 112-2115-M-001 -013 -MY3} \  and Maya Stein\footnote{Departamento de Ingenier\'ia Matem\'atica y Centro de Modelamiento
Matem\'atico (CNRS IRL2807), Universidad de Chile, Santiago, Chile. Supported by ANID Regular Grant 1260024 and by ANID grant CMM Basal FB210005.\\ Parts of this work were conceived while both authors were in residence at the Simons Laufer Mathematical Sciences Institute in Berkeley, California during the Spring semester 2025, supported by the National Science Foundation under Grant No.~DMS-1928930.}}
\date{}

\makeatletter
\let\old@bibitem\@bibitem
\def\@bibitem#1{%
  \@ifundefined{b@#1}{%
    \global\let\old@item\item
    \def\item[##1]{\global\let\item\old@item\comment}%
    \old@bibitem{#1}%
  }{%
    \old@bibitem{#1}%
  }%
}
\makeatother

\begin{document}

\maketitle

\begin{abstract}
The Erd\H os--S\'os conjecture states that every $n$-vertex graph $G$ with more than $(k-2)n/2$ edges contains every $k$-vertex tree.
    We solve the extremal cases of this conjecture, showing that for some fixed $\mu>0$, the conjecture holds for each $G$ that minimally satisfies the 
 assumptions of the conjecture and  has a subgraph~$H$ of minimum degree $\delta(H)\ge (1-\mu)k$. In our proof, we mainly have to deal with $H$ taking two different shapes: either $H$ is close to the complete graph $K_k$ or $H$ is close to the complete bipartite graph $K_{k,k}$.
\end{abstract}

\section{Introduction}

The following famous conjecture was made in the early 1960s (see~\cite{Erdos64}).

\begin{conjecture}[Erd\H os--S\'os conjecture]\label{es}
For $k, n\in\mathbb N$,  
every $n$-vertex graph with more than $(k-2)n/2$ edges contains every $k$-vertex tree.
\end{conjecture}

The conjecture is best possible. The easiest example is a complete graph on $k-1$ vertices, which has average degree $k-2$, and  contains no $k$-vertex graph, in particular no $k$-vertex  tree. More generally, any $(k-2)$-regular graph serves as an example for Conjecture~\ref{es} being tight, as it fails to contain the $k$-vertex star.  

See~\cite{maya-survey} for an overview of the 
many partial results regarding Conjecture~\ref{es}. 
The most significant  progress  today can be summarised as follows. Besomi, Pavez-Sign\'e and one of the authors~\cite{BPS3} proved the conjecture  for  large trees of linearly bounded maximum degree  and large dense  host graphs. Building on this,  Pokrovskiy~\cite{Pok24} showed the conjecture holds for all large trees of constant maximum degree. Without any restrictions on the maximum degree of the tree, the conjecture was solved for all large trees and host graphs $G$ with $|V(G)|\le (1+10^{-1})k$ in~\cite{RS25}, and very recently, Davoodi,  Piguet, {\v R}ada and Sanhueza-Matamala
gave a proof of the approximate version of the conjecture  for large dense graphs~\cite{beyond}. 
A  proof of Conjecture~\ref{es} for  large graphs was announced in the early 1990s~\cite{AKSS:ES}, but no manuscript is available.

In this paper, we focus on  the extremal case(s) of Conjecture~\ref{es}. Our main result (Theorem \ref{thm:realmain} below) solves Conjecture~\ref{es} for all graphs $G$ that minimally satisfy the
 assumptions  
and 
have a  subgraph $H$ of minimum degree almost~$k$.  
 For convenience we introduce the following definition\footnote{Robust graphs have also been called {\it strictly balanced} or {\it density-maximal} in the literature.} where $d(H)$ is the average degree of graph $H$.
\begin{definition}
\hskip-.1cm A graph $G$ is {\em robust} if 
 $d(G')<d(G)$ for each proper subgraph~$G'$~of~$G$. 
 \end{definition}
 Clearly,  any minimal counterexample to  Conjecture \ref{es} is robust. 
Our main result reads as follows. 

 \begin{theorem}\label{thm:realmain}
 There is a constant $\mu>0$ 
 such that  for each $k\in\mathbb N$  and
     each robust graph~$G$ with  $d(G)>k-2$ the following holds. If $G$  has a subgraph~$H$  with $\delta(H)\ge (1-\mu)k$, 
       then $G$ contains every tree on $k$ vertices.
 \end{theorem}
Theorem \ref{thm:realmain} will enable us to prove the dense case of Conjecture~\ref{es} in the companion paper~\cite{DenseES}:
\begin{theorem}{\normalfont\textbf{\cite{DenseES}}}
     \label{maint'}
     For every $\gamma>0$ there is an $n_0$ 
 such that for all $n\ge n_0$ and $k \ge \gamma n$,  every $n$-vertex graph $G$  with   $d(G)>k-2$ contains every tree $T$ on $k$ vertices.
 \end{theorem}

 The proof of Theorem~\ref{thm:realmain} relies on 
a  fine-tuned version of it, namely Theorem~\ref{thm:main}. There,  two possible shapes  of~$H$ are specified (bipartite and non-bipartite), as well as bounds on the order of $H$ for each case. Roughly speaking, $H$ either resembles 
the complete graph $K_k$ or  the complete bipartite graph $K_{k,k}$.
The proof of Theorem~\ref{thm:realmain} thus splits into the reduction of it to Theorem~\ref{thm:main}, and the proof of the latter theorem. 

Our paper is organised as follows. 
Section~2 contains the statement of Theorem~\ref{thm:main} and   some fundamental earlier results needed for our proofs.
We  provide an overview~of the main ideas of all proofs in Section~\ref{proofov} and show some preliminary results in Section~\ref{prelims}.  The reduction of Theorem~\ref{thm:realmain} to Theorem~\ref{thm:main} is given in Section~\ref{sec:redu}. The~proof~of Theorem~\ref{thm:main} is given in Section~\ref{mainproof6},  divided into two subsections according to the shape~of~$H$.

\subsection{Recent AI solution of the Erd\H{o}s--S\'{o}s conjecture}
It was announced very recently that GPT-6 Astra proved the Erd\H{o}s--S\'{o}s conjecture in full. Our proof was found without any use of AI. The version uploaded as the first arXiv version of this paper was ready in this form  in June 2026, but was not uploaded until the AI proof was announced. 
Although our result is now eclipsed by the AI proof, we believe that   the methods of this and the companion paper will likely have an impact on future work on related conjectures on tree containment.

\section{Theorem~\ref{thm:main} and important tools} 
 The result behind our main theorem is the following.

 \begin{theorem}\label{thm:main}
 There are $\mu>0$, $k_0\in\mathbb N$ such that  for all $k\ge k_0$ and for each robust graph $G$  with $d(G)>k-2$ the following holds. If $G$   contains a subgraph~$H$ with one of the following characteristics:
     \begin{enumerate}[(i)]
     \item\label{cliquelikelem}
     $|V(H)|
    \le (1+\mu)k$  and 
    $H$ has minimum degree $\delta(H)\ge (1-\mu)k$; or
     \item\label{bicliquelikelem}
     $H$ is  bipartite,   $|V(H)|
    \le (2+2\mu)k$ and $\delta(H)\ge (1-\mu)k$,
       \end{enumerate}
         then $G$ contains every tree on $k$ vertices as a subgraph.
 \end{theorem}
 Note that apart from describing the subgraph $H$ in more detail, Theorem~\ref{thm:main} differs from  Theorem~\ref{thm:realmain} in that  it   only holds for large graphs.  It will be easy to drop this restriction in Theorem~\ref{thm:realmain} (see the beginning of Section~\ref{sec:redu} for an argument).

The reduction of Theorem~\ref{thm:realmain} to Theorem~\ref{thm:main} relies on a result of Havet, Wood and the authors \cite{HRSW20}. 
\begin{theorem}[Theorem~1.3 in \cite{HRSW20}]
\label{thmHRSW}
 There is a $\gamma > 0$ such that if a graph has maximum degree at least $k-1$ and minimum degree at least $(1-\gamma) (k-1)$ then it contains every tree $T$ on $k$ vertices as a subgraph.\end{theorem}
 
For the proof of Theorem~\ref{thm:realmain}, we need the following two results which were  proved by the authors in~\cite{RS23a, RS23b} and in \cite{RS25}. 
 
\begin{theorem}[Theorem~1.2 in \cite{RS23b}]
\label{thmRS}
 There is a $k_0\in \mathbb N$ such that for each $k\ge k_0$ if a $k$-vertex graph has maximum degree at least $k-1$ and minimum degree at least $2 (k-1)/3$ then it contains every tree $T$ on $k$ vertices as a subgraph.
 \end{theorem}
 
\begin{theorem}[Theorem~1.2 in \cite{RS25}]
\label{thmRS25}
There are $k_0 \in\mathbb N$ and $\delta>0$ such that for all $k\ge k_0$ every graph $G$ with $|V (G)| \le (1 + \delta)(k-1)$  and with average degree exceeding $k - 1$ contains  every tree $T$ on $k$ vertices as a subgraph.
 \end{theorem}

\section{Overview of the proof}\label{proofov}

In this section we give a brief sketch of our proof. First 
we focus on how
Theorem~\ref{thm:realmain}
reduces to
Theorem~\ref{thm:main}
and then we prove 
Theorem~\ref{thm:main}. 

\subsection{Overview: reduction of Theorem~\ref{thm:realmain} to Theorem~\ref{thm:main}}

We choose $\mu'$ from Theorem~\ref{thm:realmain} much smaller than $\mu$ from Theorem~\ref{thm:main}. The case $k<k_0$ is easily checked, so we focus on the case $k\ge k_0$. Given~$H'$ as in Theorem~\ref{thm:realmain}, we need to either embed $T$ or find a subgraph of $G$ that satisfies the conditions in Theorem~\ref{thm:main} (i) or (ii). We note that if in the early stages of our embedding, we manage to 
embed $\mu k$  vertices outside every neighbourhood, then we will be able to  finish embedding the tree greedily. 

Fixing any $w\in V(H')$ and considering the set $S_w$ of all vertices that are almost complete  to $N(w)$, we show that unless one can embed $T$, 
one can extract either a dense non-bipartite subgraph of order roughly  $k$ (as in Theorem~\ref{thm:main}(i)) or a dense balanced bipartite subgraph
with parts contained in $N(w)$ and $V(H')\setminus N(w)$ (as in Theorem~\ref{thm:main}(ii)).
So if neither extraction is possible for any $w$, then for every $w$ the corresponding set $S_w$ is small. This is crucial to the fact that a random greedy embedding into~$H'$ succeeds. Indeed, 
after randomly embedding a small part of $T$, let us consider those   vertices whose neighbourhoods contain all but $\mu ' k$ of the embedded vertices.  We  show that  with high probability any two such vertices must have almost identical neighbourhoods. Therefore,  it suffices to embed many vertices outside one such neighbourhood,
which automatically protects all others and allows the final greedy completion of the embedding of $T$ in $H'$.

\subsection{Overview of the proof of Theorem~\ref{thm:main}}
We are given a robust graph $G$ of average degree exceeding $k-2$, with a certain subgraph $H$. 
For both types of $H$ we argue by contradiction to the assumption that the given $k$-vertex tree $T$ does not embed in  $G$.

\subsubsection{Bipartite $H$}

In this case,   $H$ is bipartite with sides $A_1,A_2$, having order at most $(2+2\mu)k$ and
minimum degree at least $(1-\mu)k$.
Let $T$ have colour classes $S_1,S_2$, where $|S_1|\le |S_2|$.
If $|S_1|$ is not tiny then the large minimum degree in $H$ allows a straightforward greedy
top--down embedding, so we may assume $|S_1|\le \mu k$.
Consequently, almost all vertices of $S_2$ are leaves. So, letting $L_2\subseteq S_2$ be the set of leaves in $S_2$,
the tree $T-L_2$ has size at most $2\mu k$ and maximum degree $O(\mu k)$.

Turning to the host $G$, we first isolate the vertices outside $H$ that see at least $k/8$ of $H$.
A simple edge-count using the fact that only few vertices of~$H$ can have degree $k-1$ or more in $G$
shows that $|X|=O(\mu k)$, and hence the remainder $G-H-X$ has minimum degree $\Omega(k)$.
We then show that any vertex with  many neighbours in $H$ and many neighbours in $G-H-X$
must have comparatively low total degree, as otherwise we can embed a large part of $T$ outside $H$ and
finish inside $H$.

Next, we partition the high-degree vertices of $X$ into $B_1,B_2$ according to which side of $H$
they mainly connect to, and consider the bipartite graph $H'$ induced by the edges between
$A_1\cup B_1$ and $A_2\cup B_2$.
A weighted averaging argument (based on minimality/robustness) produces large sets of vertices on each
side of $H'$ with substantial ``embedding potential'' measured by a function $f$ that counts neighbours
available inside $V(H)\cup X$ and, with extra weight, neighbours in $G-H-X$.
Now, if we embed an appropriate separator vertex of $T$ and a
controlled family of components of $T$ into $G-H-X$, then the remainder of $T-L_2$ embeds top--down
in $H'$, and the remainder of~$L_2$ can be attached greedily using the degree
guarantees encoded in $f$.

\subsubsection{Non-bipartite $H$}
Our starting point is a dense subgraph $H\subseteq G$ with $|H|\le (1+\mu)k$ and $\delta(H)\ge (1-\mu)k$.
We   enlarge $H$ to a set $H'$ by repeatedly adding vertices outside the current set which have at least
$\frac{3+100\mu}{4}k$ neighbours inside. Each such addition increases the average degree by a positive constant
(cf.~\eqref{increase}). Consequently, either we add $2\mu k$ vertices and obtain a graph of average degree at least $k$,
in which case $T$ embeds by Theorem~\ref{thmRS25}, or else $H'$ is ``closed'' in the sense that every vertex
outside $H'$ has fewer than $\frac{3+100\mu}{4}k$ neighbours in~$H'$ (cf.~\eqref{outsideH'baddeg}).
We may further assume that $H'$ contains no vertex of degree at least $k-1$, since otherwise we embed $T$
using Theorem~\ref{thmRS} applied to the closed neighbourhood of such a vertex.

We construct an exceptional $X$ set outside $H'$ by iteratively including  vertices with tiny degree to the remaining outside vertices. Then either
$|X|$ is  large enough to force a vertex of degree at least $k-1$ in $H'':=G[H'\cup X]$ (again giving an embedding), or else
the leftover graph $G':=G-H'-X$ has linear minimum degree (cf.~\eqref{mindegggforreal}).
This means that the $k/2$-separator $t$ of $T$  cannot have many leaf-children
(Claim~\ref{thasatmostmukleafch}), as otherwise we could use a vertex of $H'$ with a suitably large weighted neighbourhood
to embed the neighbourhood of $t$ and then complete the embedding greedily.

We then define a near-clique   $H^*$ inside $H'$ which has size close to $k$ and minimum degree $k-O(\mu k)$ (and therefore strong
common-neighbour properties) and we control the number of large degree vertices of  $H''$ (Claims~\ref{Aug14claim1} and~\ref{Aug14claim2}), which leads to the observation that $T$ cannot have $100\mu k$ leaves
(Claim~\ref{Tatmost200leaves}), as this would yield a matching-and-expansion based embedding.
Thus $T$ has few leaves, which implies that most vertices lie on long bare paths. In the final step we route
the bare paths through the small exceptional parts of the host
and finish the embedding in $H^*$.

\section{First steps}
\label{prelims}

\subsection{Robust graphs}
We start with some basic observations.
The 
first observation holds by the definition of robustness and considering the subgraph~$G-S$:
 
\begin{observation}\label{minimalityOfG}
For each robust  $G$ with average degree   $d(G)>k-2$ and for each $S\subseteq V(G)$,  the total number of edges having at least one endpoint in $S$  exceeds $\frac{k-2}2 |S|$.
\end{observation}

Applying Observation~\ref{minimalityOfG} to a single vertex of $G$ to obtain:
\begin{observation}\label{minDegofG}
The minimum degree of a robust graph
 of average degree exceeding $k-2$ is at least $\lceil \frac{k-1}{2}\rceil$.
 \end{observation}

 Next, we see that a tree with many leaves sharing a parent is easy to embed.

\begin{observation}\label{fewleaves}
Let $T$ be a $k$-vertex tree  having a vertex $v$ adjacent to at least $\frac k2$ leaves. Then $T$ is contained  in each robust graph $G$ of average degree exceeding $k-2$.
\end{observation}
Observation \ref{fewleaves} holds since we can embed $v$ into a maximum degree vertex of $G$ and then embed  $T-L$, where $L$ are the leaves at $v$, greedily (which is possible as $G$, being robust, has minimum degree at least $(k-1)/2$). Finally, we embed $L$.

\subsection{Basic facts about  trees}

The following definition is crucial for splitting any given tree into components whose sizes we can control.
\begin{definition}[$\alpha k$-separator]\label{sepa}
For $\alpha\in [\frac 12,1)$, we call a vertex $t$ of a $k$-vertex tree $T$ an {\em $\alpha k$-separator} for $T$ if each component of $T-t$ has at most $\alpha k$ vertices and all but at most one of these components have at most $(1-\alpha) k$ vertices.
\end{definition}

We will make extensive use of the   folklore observation that an $\alpha k$-separator always exists    (for a reference see e.g.~\cite{HRSW20}).
\begin{lemma}\label{lem:sepa}
For each $\alpha\in [1/2,1)$, each $k$-vertex tree $T$ has a $\alpha k$-separator.
\end{lemma}

We need a few more  lemmas on trees and separators.

\begin{lemma}
\label{subsetofcompsofT2}
 Let $k\in\mathbb N$ and $b,d\in\mathbb R$  with $b>0$ and $0\le d< \frac b4$. Let $T$ be a $k$-vertex tree, let $t\in V(T)$, and let  
$\mathcal B$ be the set of all nonsingleton components of $T-t$ on at most  $b$ vertices. If $|\bigcup\mathcal B|\ge 2d$, 
 then there is a set  $\mathcal C\subseteq \mathcal B$   such that $|\mathcal C|\le \lceil d\rceil$ and $2d \le |V(\bigcup \mathcal C)|\le b$.
\end{lemma}
\begin{proof} 
If $d=0$, take $\mathcal C=\emptyset$. Hence assume $d>0$. 
Order the components of $\mathcal B$ in decreasing order with respect to their size as  $C_1,$ $C_2, \ldots, C_\ell$. 
 Let $i\ge 1$ be minimum with $c:=|V(C_1\cup C_2\cup\ldots\cup C_i)|\ge 2d$.  Then $i \le \lceil d\rceil$  as the $C_j$ all have size at least 2. 
  So we are done unless $c>b$ which implies  $i\ge 2$,
  $|C_i|>b-2d\ge\frac b2$ and $|C_{i-1}|<2d\le\frac b2$, a contradiction to how we ordered~the~$C_i$.
\end{proof}

\begin{remark}\label{theremark}
    Note that the condition $|\bigcup\mathcal B|\ge 2d$ in Lemma~\ref{subsetofcompsofT2} holds if $N(t)$ contains at most $k-2d-1$ leaves and no component of $T-t$ has size exceeding $b$. 
\end{remark}

\begin{lemma}\label{180p}
Let $d, k\in\mathbb N$ satisfy $ d<\frac k{1000}$, let $T$ be a $k$-vertex tree with partition classes $S_1, S_2$ and set of leaves $L$, and let $t\in V(T)$. 
 Suppose  
that (a) $|S_2\cap L|\ge k-\frac k{1000}$,  (b) no component of $T-t$ has more than $\frac k{1000}$ vertices, and (c) $|N(t)\cap L|\le \frac k2$. 
    Then there is  a set $\mathcal C$ of  nontrivial components of $T-t$ such that $|\mathcal C|\le d$ and $100 d
    \le 
    |V(\bigcup \mathcal C)\cap S_2\cap L|
    \le |V(\bigcup \mathcal C)|\le \frac k{4}$.
\end{lemma}

\begin{proof}
If $d=0$, take $\mathcal C=\emptyset$. Hence we may assume
that $d\ge 1$.
As each nonsingleton component contains at least one vertex from $S_1$,   (a) implies  that $T-t$ has at most $\frac k{1000}$ nonsingleton components.  So, by (a) and (c), these components contain at least $\frac k2-\frac k{500}=\frac{498}{1000}k$ vertices of $S_2\cap L$. Thus on average, each nonsingleton component contains at least $498$ vertices of $S_2\cap L$. 
So we can choose a subset $\mathcal C'$ of at most $d$ of these components such that $|V(\bigcup\mathcal C')\cap S_2\cap L|\ge 498 d$ (note that if there are less than $d$ components we can take all of them). If $|\bigcup\mathcal C'|\le \frac k4$, we are done by taking $\mathcal C=\mathcal C'$, so assume otherwise. We obtain $\mathcal C$ from $\mathcal C'$ by deleting components until $|\bigcup\mathcal C|\le \frac k4$. By (b), $|\bigcup\mathcal C|\ge \frac k4-\frac k{1000}$, and hence by (a), $|V(\bigcup\mathcal C)\cap S_2\cap L|\ge  \frac k4-\frac k{1000}- \frac k{1000}\ge 100 d$. 
\end{proof}

 Only for the next lemma and its applications, we need a quick notation: for a tree~$T$ and root $t$, let $T_O$ be the union of the odd levels and let $T_E$ be the union of its even levels. So  $T_O$ and $T_E$ are the two colour classes of $T$, $t \in T_E$,  and  $V(T)=T_O\cup T_E$.

\begin{lemma}[Odd and even levels lemma]\label{oddandeven}
 Let $T$ be a $k$-vertex tree, let $t$ be a $\frac k2$-separator for $T$ with $d(t)\le 2d$ for some $1\le d<\frac k{10}$, and let $\mathcal C$ be the set of components of $T-t$. 
 Then there is a set $\mathcal C'\subseteq\mathcal C$   such that the union of  $T_O\cap V(\bigcup\mathcal  C')$ and  $T_E\cap V(\bigcup (\mathcal C\setminus\mathcal  C'))$ has size between $d$ and $k-1-d$.
\end{lemma}
\begin{proof}
By Remark~\ref{theremark} and since $d(t)\le 2d\le \frac k2-2d-1$, we can apply
Lemma~\ref{subsetofcompsofT2} 
 with $b=\frac k2$ and $d$. This gives a set   $\mathcal D\subseteq \mathcal C$ 
with 
$2d \le |V(\bigcup \mathcal D)|\le k/2$. Observe that some  $T_1\in \{ T_E, T_O\}$ contains at least $d$ vertices of $\bigcup\mathcal D$, and some  $T_2\in \{ T_E, T_O\}$ contains at least~$d$ vertices of $\bigcup(\mathcal C\setminus \mathcal D)$. If $T_1\neq T_2$ we  set $\mathcal C':=\mathcal C$, and if $T_1= T_2$ we  set~$\mathcal C':=\mathcal D$.
\end{proof}

\subsection{Two useful definitions}

\begin{definition}\label{topdown}
Embed a tree $T$ in a graph $G$ in a {\em top-down fashion} means that we start the embedding at the given root of $T$  and then  embed the other vertices successively, at each step embedding one whose parent has already been embedded.
\end{definition}
Definition~\ref{topdown} is only meant to establish the order in which we process the vertices of $T$. Exact instructions on where to embed each vertex of $T$ will be given in the proofs.

\begin{definition}\label{def:VmS}
Let $H$ be a subgraph of  a graph $G$.
For a subset $S$ of $V(G)$ and for $m\in\mathbb N$, let $V_H(m,S)$ be the set of all vertices of $H$ having degree at least $m$ into $S$.
\end{definition}

\section{Reduction of Theorem~\ref{thm:realmain} to Theorem \ref{thm:main}}\label{sec:redu}

In this section we will show that Theorem \ref{thm:realmain} follows from Theorem~\ref{thm:main}.
For this, 
let $\mu$ and $k_0$ be given by Theorem~\ref{thm:main}. Replacing $k_0$ by a larger integer
if necessary, we may assume that, for every $s\ge k_0$,
$(\frac{\mu^{15}s^2}{10^4})^2
e^{-\mu^{12}s/8}<\frac{\mu^6}{2}.$ Define $\mu'$  for Theorem \ref{thm:realmain}  as 
 $$\mu':=\min\Big\{\frac{\mu^{15}}{10^4}, k_{0}^{-1}, \gamma,\frac12\Big\}$$ where~$\gamma$ is given by Theorem~\ref{thmHRSW}. 
The premises of Theorem \ref{thm:realmain} give a $k$-vertex tree~$T$, a robust graph~$G$  of average degree  $d(G)>k-2$, and a subgraph $H'$ of $G$ with 
$\delta(H')\ge (1-\mu')k$. 
If $k<k_0$, then  $\delta(H')\ge (1-\mu')k\ge k- k_{0}^{-1}k> k-1$ and we can greedily embed~$T$  in~$H'$. 
 So we can assume $k\ge k_0$.
We will show that either $T\subseteq G$   or  $G$ contains a subgraph~$H$   as in Theorem~\ref{thm:main}.

Consider any $w \in V(H')$ and write $N(w)$ for its neighbourhood in $H'$. We may assume that 
\begin{equation}\label{Nw}
|N(w)|<k-1, \ \text{and thus} \ \Delta(H')<k-1,
\end{equation}
as otherwise  Theorem~\ref{thmHRSW} 
ensures that $H'$ contains $T$.
Let $S_w:=V_{H'}\big((1-\frac{\mu^2}{8})k,\, N(w)\big)$.
If $S_w\setminus N(w)$ contains a set $S'_w$ of size $(1-\frac{3\mu}{4})k$, then
let $N'\subseteq N(w)$ contain all but the $\mu k/2$ vertices of lowest degree into $S'_w$.
We claim that $N'$ has minimum degree at least $(1-\mu)k$ into~$S'_w$.
Indeed, otherwise some vertex of $N'$ has fewer than $(1-\mu)k$ neighbours in~$S'_w$, and hence
each of the $\mu k/2$ deleted vertices has fewer than $(1-\mu)k$ neighbours in $S'_w$.
Thus the number $e(\overline{H'}[N(w),S'_w])$ of missing edges between $N(w)$ and $S'_w$ is at least
\[
\frac{\mu k}{2}\big(|S'_w|-(1-\mu)k\big)=\frac{\mu k}{2}\cdot \frac{\mu k}{4}=\frac{\mu^2k^2}{8}.
\]
On the other hand, by the definition of $S_w$, every vertex of $S'_w$ has at least
$(1-\frac{\mu^2}{8})k$ neighbours in $N(w)$, and since $|N(w)|<k$ this gives
\[
e(\overline{H'}[N(w),S'_w])\le |S'_w|\cdot \frac{\mu^2k}{8}<\frac{\mu^2k^2}{8},
\]
a contradiction. Hence $\delta(N',S'_w)\ge (1-\mu)k$.
Moreover, each vertex of $S'_w$ has degree at least
$(1-\frac{\mu^2}{8})k-\mu k/2\ge (1-\mu)k$ into $N'$.
 As the bipartite subgraph of $H'$ with vertex classes $N'$ and $S'_w$
and edge set $E_{H'}(N',S'_w)$
also has at most $(2+2\mu)k$ vertices, we see that it is as  $H$ in Theorem~\ref{thm:main} (ii). 
 
 Similarly, if 
 $|S_w\cap N(w)|\ge (1-3\mu/4)k$, then 
$|N(w)\setminus S_w|<k-1-(1-\frac 34\mu)k < \frac 34\mu k$ by~\eqref{Nw}, 
and thus
$G[S_w\cap N(w)]$ has minimum degree at least 
 $(1-\frac{\mu^2}8)k-|N(w)\setminus S_w|\ge(1-\mu)k.$ 
 So by~\eqref{Nw}, $G[S_w\cap N(w)]$ is as  $H$ in Theorem~\ref{thm:main} (i).

 Thus we can assume that for each $w \in V(H')$ we have
\begin{equation}\label{Swmax}
\max\{|S_w\cap N(w)|, |S_w\setminus N(w)|\}< (1-\frac{3\mu}4)k.
\end{equation}

Furthermore, for each $w \in V(H')$, we may assume that
 \begin{equation}\label{Swmin}
\min\{|S_w\cap N(w)|, |S_w\setminus N(w)|\}< \frac{\mu k}2,
\end{equation}
as otherwise, by the definition of $S_w$, each of the at least $\mu k/2$ vertices of $S_w\setminus N(w)$ sees at least half of the at least $\mu k/2$ vertices of $S_w\cap N(w)$. Thus $S_w\cap N(w)$ must contain a vertex $v$ that sees at least half of $S_w\setminus N(w)$. That is, $v$ has at least $\mu k/4$ neighbours outside $N(w)$. So since $v\in S_w$, it has degree at least 
$(1-\mu^2/8)k + \mu k/4\ge $
$k-1$ in $H'$, which is impossible by~\eqref{Nw}.

Together, \eqref{Swmax} and \eqref{Swmin} imply that for each $w \in V(H')$, we have
\begin{equation}\label{Swsmall}
   |S_w| \le (1-\frac{\mu}4)k. 
\end{equation}

We root $T$ at a  vertex  $t_0$  with the most leaf children. 
For each non-singleton component of $T-t_0$ we order its vertices depth-first, starting at its root, so that for each vertex $v\neq t_0$, its parent appears before $v$ in the ordering, and all  leaf children of $v$ appear directly after $v$ in the ordering.  This gives a (not uniquely defined) ordering of the vertices of each nonsingleton component of $T-t_0$. We order  the nonsingleton components of $T-t_0$ nondecreasingly according to their size, and concatenate the orderings of their vertices. Adding the leaves of $T-t_0$ to the end, we obtain an ordering $t_1, t_2, \ldots, t_{k-1}$ of   $V(T-t_0)$. 
We will embed $V(T-t_0)$ in this order, in the following way.

   We embed $t_0$ in arbitrary vertex $w_0$ of $H'$ and successively embed the vertices $t_i$ of $T-t_0$ in $H'$ as follows until step $$\ell:=\mu^2 k/10.$$ In each step $i\le\ell$ 
   we choose
    the image of $t_i$  randomly among all unused neighbours $w_i$ of the image of the parent of $t_i$.

    Let  $U$ be the set of all vertices of $H'$ that were used in these first $\ell$ steps, and let $U'$ be the set of all vertices of $H'$ that were used in these first $$\ell':=\mu' k$$ steps of our embedding. As usual, let $N(U')=N_{H'}(U')$ be the set of all vertices of $H'$ (including those from $U'$) having at least one neighbour in $U'$. Note that 
    \begin{equation}\label{NU}
        |N(U')|\le \ell'\Delta(H')< \mu'k^2
    \end{equation}
     by~\eqref{Nw}.
  Call a vertex $w\in N(U')$ {\it endangered} if  $|U\cap N(w)|>\ell-\mu'k$.
Intuitively, $w$ is endangered if after the first $\ell$ random steps we have used almost all of $N(w)$,
so in the last $\mu'k$ steps we might run out of neighbours of $w$ unless we have embedded many vertices outside $N(w)$ earlier.

\begin{claim}\label{aboutthepairs}
    The probability $p^*$ that there are distinct $v,w\in N(U')$ with $|N(v) \cap N(w)| < (1-\mu^8 )k$ which are both endangered at step $\ell$ satisfies
     $p^*<\mu^6/2$.
\end{claim}
      \begin{proof}
Condition on the first $\ell'$ steps of the embedding, and let us
estimate the probability $p_{v,w}$ that fixed $v,w\in N(U')$  with  $|N(v) \cap N(w)| < (1-\mu^8)k$ are  both   endangered.
For this, observe that if both $v$ and $w$ are  endangered, then  $|U\setminus N(v)  |< \mu' k$  and  $|U\setminus N(w)  |< \mu' k$ and therefore $$|U\setminus ( N(v) \cap N(w))|< 2\mu' k.$$ 
Our hypotheses imply that  for any vertex $x$,
$$|N(x)\setminus  ( N(v) \cap N(w))|>
(1-\mu')k-|N(v)\cap N(w)|
\ge 2\mu'k+ \mu^9k.$$
 So, letting $X_j$ be i.i.d.~random variables  each of which is 1 with probability~$\mu^9$    and 0 otherwise, 
 the preceding
bound and~\eqref{Nw} imply by stochastic domination that
$p_{v,w}$ is bounded from above by the probability that $
\sum_{\ell'<j\le\ell}X_j<2\mu'k$.
(Indeed, as long as fewer than $2\mu'k$ vertices have been
embedded outside $N(v)\cap N(w)$, the conditional probability
that the next vertex is embedded outside this intersection is at
least $\mu^9$.) Moreover,
$\mathbb E[\sum_{\ell'<j\le\ell}X_j]
=\mu^9(\ell-\ell')
\ge\mu^{12}k
\ge100\mu'k$
 and hence  $$p_{v,w}\ \le \ \mathbb P\big [\sum_{\ell'< j\le \ell} X_j< 2\mu' k   \big] \ \le \        e^{-\mu^{12}k/8},$$  
 by an easy application of Chernoff's inequality.    Thus, by~\eqref{NU},
\[
p^*
\le
(\mu'k^2)^2e^{-\frac{\mu^{12}k}{8}}
\le
(\frac{\mu^{15}k^2}{10^4})^2
e^{-\frac{\mu^{12}k}{8}}
<\frac{\mu^6}{2},
\]
where the last inequality follows from our choice of $k_0$.
      \end{proof}

\begin{claim}\label{w0notendangered} 
    The probability $p^*_0$ that $w_0$ is endangered satisfies $p^*_0\le 1-\mu^6$.
\end{claim}
\begin{proof}
    Note that Observation~\ref{fewleaves} implies that  $T-t_0$ has at most ${k}/{2}$ singleton components.  By our choice of $\ell$ and of the ordering of the vertices of $T$, the vertices $v_1\ldots, v_{\ell}$ do not meet these singleton components. Moreover, at least $\ell/3$ of these vertices are not roots
of components of $T-t_0$. 
  By~\eqref{Nw} and~\eqref{Swsmall}, for each such non-root vertex $v$, the probability that its parent
is not embedded in $S_{w_0}$ is at least
    $$\frac{(1-\mu' )k-|S_{w_0}|-\ell}k
    \ge \frac \mu 4-\mu'-\frac{\mu^2}{10}
    \ge \frac\mu 5.$$ 
    Furthermore, if $v$'s parent is not embedded in $S_{w_0}$, then by the definition of the set  $S_{w_0}$,  the probability $v$ is not embedded in $N(w_0)$ is at least $$\frac{(1-\mu')k-(1-\frac{\mu^2}8)k-\ell}k\ge  \frac{\mu^2}{50}.$$
    
     Thus, the expected number of vertices from $U$  not embedded into $N(w_0)$ is at~least $\frac{\ell}{3}\cdot\frac{\mu}{5}\cdot\frac{\mu^2}{50}
\ge \mu^4\ell$. So the expected number of vertices  embedded into $N(w_0)$ is at most $\ell-\mu^4\ell$. 
Hence by Markov's inequality, the probability that more than $\ell - \mu' k$ vertices were embedded in  $N(w_0)$ is at most 
$\frac{\ell-\mu^4\ell}{\ell - \mu' k}\le 1-\mu^6,$
which proves the claim.
\end{proof}

      By Claims~\ref{aboutthepairs} and~\ref{w0notendangered}, there is an embedding of $T':=T[\{t_0,t_1,\ldots, t_{\ell}\}]$  such that
      \begin{enumerate}[(a)]
      \item $w_0$ is not endangered and
      \item\label{b}  if $v,w$ are endangered then $|N(v) \cap N(w)| \ge (1-\mu^8)k$. 
       \end{enumerate}
      We will from now on fix such an embedding of $T'$.
      Note that we may assume that
\begin{equation}\label{nomukleaves}
\text{no vertex of $T$   is adjacent to $\mu' k$ or more leaves,}
\end{equation}
as otherwise our choice of $t_0$ implies that 
$t_0$ is adjacent to at least $\mu' k$ leaves. We
  can embed the remainder of  $T$ except for these $\mu' k$ leaves greedily into~$H'$, and finish by embedding these leaves, which is possible by~(a).

Moreover,
      if $H'$ has no endangered vertices then  the high minimum degree of $H'$ allows us to complete the embedding of $T$ in $H'$ greedily. 
      Also, if for each endangered vertex $w$ we can ensure that before we embed the last $\mu' k$ vertices of $T$, we have embedded at least $\mu' k$ vertices outside of $N(w)$, then we can complete the embedding of $T$ in~$H'$ greedily.
This is what we will do in the remainder of the proof. We will embed the  vertices $v_{\ell+1},\ldots,v_k$ in this order, at each step randomly choosing an image uniformly among the unused neighbours of the image of the parent, as before. It only remains to verify that after step $k-\mu' k-1$, with positive probability we have that for each endangered $w\in V(H')$
\begin{equation}\label{foreachwweembed}
\text{at least $\mu' k$ vertices were embedded outside of $N(w)$.}
\end{equation}

For this, first note that our ordering of the vertices of $T$  ensures that there is at most one vertex $t\in V(T')\setminus\{t_0\}$ such that $T-T'$ contains leaf children of $t$. 
So, by~\eqref{nomukleaves}, none of the vertices in $\{t_{\ell+\mu' k},\ldots, t_{k-\mu' k}\}$ are leaf children of any vertex from $T'$, and furthermore, at least half of these vertices  are not  children  of   vertices from $T'$, which means that their parents lie in $T-T'$. Letting $C$ denote  the set of all vertices of $T-T'$ with parents in $T-T'$, we deduce that $$|C|\ge \frac{|T-T'|}3 \ge \frac k4.$$

Consider a fixed endangered vertex ${w^*}$. 
 Similarly as in the proof of Claim~\ref{w0notendangered} we see that  for each of the first $\mu^2k/1000$ vertices of $C$, which by our ordering are embedded before step $\ell+\ell'+2\mu^2k/1000+1<\mu^2k/9$, the probability that its parent $p$ is not embedded
    in~$S_{w^*}$ is at least $\mu/5$ by~\eqref{Nw} and~\eqref{Swsmall}, and if $p$ is not embedded
    in $S_{w^*}$, then the probability $c$ is not embedded in $N({w^*})$ is at least~$\mu^2/100$. 
     Thus the expected number of vertices from $C$ embedded outside $N({w^*})$ is at least $\mu^5k/10^6$, and the expected number of vertices from $T-T'$ embedded  in $N({w^*})$ is at most $|T-T'|-\mu^5k/10^6$. 
     Then 
     by Markov's inequality, the probability that more than $|T-T'| - 2\mu^8k$ vertices of $T-T'$ were embedded in  $N({w^*})$ is at most 
$$\frac{|T-T'|-\mu^5k/10^6}{|T-T'| - 2\mu^8 k}<1.$$ 

Thus with positive probability 
\begin{equation}\label{atleast200}
\text{at least $2\mu^8 k $ vertices were embedded outside $N({w^*})$, }
\end{equation}
which we  assume from now on. In particular, \eqref{foreachwweembed} holds for $w^*$. Now consider any other endangered vertex $w$. By~\eqref{b}, we know that $|N(w) \cap N(w^*)| \ge (1-\mu^8)k$ and thus by~\eqref{Nw}, $|N(w) \setminus N(w^*)|\le \mu^8k$. So, by~\eqref{atleast200}, $\mu^8 k > 2\mu' k $ vertices were embedded outside $N({w})$, which proves~\eqref{foreachwweembed} for~$w$. This finishes the embedding of $T$ in $H'$.

\section{Proof of Theorem \ref{thm:main}}
\label{mainproof6}
Constants $\mu$ and $k_0$ are chosen as follows. Let $\mu:= \min\{10^{-10}, \gamma, \delta^{10}\}$, where $\gamma$ is from Theorem~\ref{thmHRSW} and $\delta$ is from Theorem~\ref{thmRS25}. Set $k_0:=\max\{10^{100}\mu^{-1}
,k_0', k_0''\}$ where $k_0'$ is from Theorem~\ref{thmRS} and $k_0''$ is from Theorem~\ref{thmRS25}. We will work with $k\ge k_0$.
We split the proof into two parts, according to the type of subgraph~$H$ the graph $G$ contains.

\subsection{Bipartite $H$}\label{sec:balbip}
In this section we prove Theorem \ref{thm:main} for subgraphs $H$ as in Theorem \ref{thm:main} (ii). That is, we are given a robust graph $G$ of average degree exceeding $k-2$ and a bipartite subgraph $H$ with $|V(H)|
    \le (2+2\mu)k$ and $\delta(H)\ge (1-\mu)k$.  Let $A_1$, $A_2$ be the two sides of $H$.
  We have to embed in $G$ a given $k$-vertex tree $T$, with sides $S_1$ and $S_2$. Set $s_i=|S_i|$ for $i=1,2$. We can assume that $s_1\le s_2$.  For contradiction,  assume $T$ cannot be embedded in $G$.
    
In particular, we cannot embed  $T$ greedily in a top-down manner, and thus,  
\begin{equation}\label{s1}
s_1 \le \mu k.
\end{equation}
Let $L_2$ be the set of all leaves of $T$ in $S_2$.   By~\eqref{s1}, and since (rooting $T$ at any vertex of $S_2$,) each non-leaf in $S_2$ has a child in $S_1$,
\begin{equation}\label{T'andL2}
\text{$\Delta(T-L_2)\le s_1\le \mu k$  and $|S_2\setminus L_2|\le s_1\le \mu k$}. 
\end{equation}
 
Suppose $T$ has a vertex $x$ of degree at least $k/2+s_1$. Then $x\in S_1$, and by~\eqref{T'andL2} there are at least $d(x)-s_1\ge k/2$ leaves adjacent to $x$ and we are done by Observation~\ref{fewleaves}.  
Therefore, we may assume that
\begin{equation}\label{DeltaT}
\Delta(T)< \frac k2+s_1\le \frac k2+\mu k,
\end{equation}
where we used~\eqref{s1} for the inequality.
We let $t$ and $t'$ be the $\frac k2$-separator and the $\frac 34k$-separator for $T$ from Definition~\ref{sepa}. Observe that since each non-singleton component contains at least one vertex of $S_1$, it follows that 
\begin{equation}\label{T-tfewcomps}
\text{
 $T-t$ and $T-t'$ each have at most $s_1$ non-singleton components.}
 \end{equation}

Let us now turn to $G$. 
Since $\delta(H)\ge (1-\mu)k$, we know that for $i=1,2$, 
\begin{equation}\label{Ai1+3}
|A_i|\le (1+3\mu )k, 
\end{equation}
and so each vertex in $A_i$ 
sees all but at most $4\mu k$ vertices of $A_{3-i}$. 
We claim that for $i=1,2$, $A_i$ contains 
at most $5\mu k$ vertices of degree  at least $k-1$ in~$G$, that is
\begin{equation}\label{highAi}
|V_{A_i}(k-1, V(G))|\le 5\mu k.
\end{equation}
Assume this is not the case for some $i$. 
We can then embed $T-L_2$  in the bipartite graph $H$ greedily, mapping $S_1$  to $V_{A_i}(k-1, V(G))$  and mapping  $S_2\setminus L_2$ to $A_{3-i}$.  We finish by embedding $L_2$ greedily. 
So~\eqref{highAi} holds. 

By (\ref{highAi}), we know that the number of edges between $H-V_{H}(k-1, V(G))$ and $$X:=V_{G-H}(k/8,V(H))$$ is bounded from below by
$$|X|\cdot (k/8-10\mu k)\ge |X|\cdot k/9$$
and bounded from above by
$$|H|\cdot (k-1-\delta(H))\le  (2+2\mu )k\cdot \mu k\le  2.1k\cdot \mu k.$$

It follows that
\begin{equation}\label{HighX}
|X|\le 19\mu k.
\end{equation}

Hence, since $G$ has minimum degree at least $\frac{k-1}2$ by Observation~\ref{minDegofG} and since $k\ge (2\mu)^{-1}$, we obtain that
\begin{equation}\label{minG-H-X-new}
\delta(G-H-X)\ge \frac{k}2-19\mu k- \max\{d(v,H):v\in G-H-X\}.
\end{equation}
Since by definition of $X$, any vertex $y\in V(G-H- X)$ has less than $\frac k8$ neighbours in $H$, it follows that
\begin{equation}\label{minG-H-X}
\delta(G-H-X)>\frac{k}{3}.
\end{equation}

Let $Y$ be the set of all vertices of $G$ which have degree at most $\frac{k}{2}+50\mu k$ in~$G$. 
We claim that  
\begin{equation}\label{Low}
Y':=V_G\big (3\mu k, V(H)\big )\cap V_G\big (\mu k, V(G-H-X)\big )\subseteq Y.
\end{equation}

Assume that, to the contrary, there is a vertex $v\in Y'\setminus Y$. If $T-t'$ has more than $\mu k$ singleton components, then by~\eqref{s1} these consist of vertices from~$S_2$ (and $t'\in S_1$), and we let $\mathcal K$ be a set of exactly $\mu k$ singleton  components of $T-t'$.
Otherwise,  we let 
 $\mathcal K$ be  a set of at most $\mu k$ nonsingleton components of $T-t'$ such that $2\mu k\le |\bigcup\mathcal K|\le k/4$. 
 Such a set exists 
by  Lemma~\ref{subsetofcompsofT2} with $d=\mu k\ge 1$ and $b=k/4$, which we can apply  since $t'$ is a $\frac 34k$-separator and thus more than $k-1-\frac 34k\ge 3d$ vertices lie in nonsingleton components of $T-t'$ of size at most $\frac k4$, and furthermore, by~\eqref{s1} and~\eqref{T-tfewcomps} there are at most $\lceil\mu k \rceil$ singleton components.

We  embed~$t'$  in~$v$, and embed the  (at most $\lceil\mu k\rceil$)  roots of components  of~$\mathcal K$  in $G-H-X$. 
By~\eqref{minG-H-X}, we can ensure that all of $\bigcup\mathcal K$ is embedded in $G-H-X$.

We next embed all singleton components of $T-t'$ in $G$, avoiding a set~$N$ of $\mu k$ neighbours of $v$ in $H$. If $\mathcal K$ consists of nonsingleton components, then  this  is possible   since $v\notin Y$, and $T-t'$ has at most $\mu k$ singleton components. If $\mathcal K$ consists of singleton components, then this is possible  since $v\notin Y$ and since by Observation~\ref{fewleaves}, $T-t'$ has at most $k/2$ singleton components. 

We  embed the (at most $\mu k$) roots of the nonsingleton components of $T-t'$ in~$N$, and
 finally embed the nonsingleton components of $T-t'$ greedily in~$H$. This is possible  because H has minimum degree $(1-\mu)k$, and  we have already embedded at least $\mu k$ vertices   outside $H$. 
This proves~\eqref{Low}.

We claim that 
\begin{equation}\label{thereisanAi}
\text{for each
 $v\in (V(H)\cup X)\setminus Y$ there is an $i$ with  $|N(v)\cap A_i|<\mu k$.}
 \end{equation}
 Indeed, otherwise there is a  $v\in (V(H)\cup X)\setminus Y$ with at least $\mu k$ neighbours in each $A_i$. By~\eqref{DeltaT} and since $v\notin Y$, we know that $d(t)\le\Delta(T)\le k/2+\mu k\le d(v)$. We embed $t$ in $v$.  We note that by~\eqref{s1} and~\eqref{T-tfewcomps},  the set $R\subseteq N(t)$ of all roots of nonsingleton components of $T-t$ has size  at most $\mu k$.
 
 If $d(t)\ge 2\mu k$  we embed $N(t)$ in $G$, putting at least $\mu k$ of these vertices into each~$A_i$ (which is possible as we assume $v$ has at least $\mu k$ neighbours in each $A_i$), and ensuring that all of $R$ is embedded in $H$. We can then finish the embedding greedily in the  bipartite graph~$H$.

So assume $d(t)< 2\mu k$, and let  $L$ be the set of leaves adjacent to $t$.  By Lemma~\ref{oddandeven}, 
 we can partition $R\cup L$ into $R_1$ and~$R_2$ so that the union of the odd levels of components with root in $R_1$ and  the even levels of components with root in $R_2$ has size between $\mu k$ and $(1-\mu )k$. 
 We embed $R_i\setminus L$   into $A_i$, for $i=1,2$, embed $R_i\cap L$   into $A_i$ if possible, and otherwise arbitrarily into $G$,  and can then finish greedily as before, using at least $\mu k$ and at most $(1-\mu )k$ vertices  in each $A_i$. This proves~\eqref{thereisanAi}.

We note that  
\begin{equation}\label{nomorethan21}
\text{\ \hskip-.352cm for each $x\in X\setminus Y$ there is an $i$ with $|N(x)\cap A_{3-i}|\ge \frac{k}{2}+28\mu k$.\hskip-.1cm}
 \end{equation}
 Indeed, each vertex  $x\in X\setminus Y$ has at least $k/2+50\mu k$ neighbours, while sending  at most $\mu k$ edges to $G-H-X$  by~\eqref{Low} 
 and at most $19\mu k$ vertices to~$X$ by~\eqref{HighX}. So, as $x$ sends at most 
 $\mu k$ edges to one of the $A_i$ by~\eqref{thereisanAi}, it sends at least $\frac{k}{2}+(50-1-19-1)\mu k>\frac{k}{2} +28\mu k$ edges to the other $A_i$. This proves~\eqref{nomorethan21}.
 
For $i=1,2$, 
we let $B_i$ be the set of all vertices in $X\setminus Y$ that see at least $k/2+28\mu k$ vertices of 
$A_{3-i}$ and at most $\mu k$ of $A_i$.
By~\eqref{thereisanAi} and~\eqref{nomorethan21}, we know that
\begin{equation}\label{B1B2parti}
\text{$B_1$ and $B_2$ constitute a partition of $X\setminus Y$.} \end{equation}

\noindent  For $i=1,2$, let $v_i\in A_i \cup B_i$ be any vertex
maximising $|N(v)\cap V(G-H-X)|$ and set $$d'_i:=|N(v_i)\cap V(G-H-X)|.$$
Since $v_i\in (V(H)\cup X)\setminus Y$, it has more than $3\mu k$ neighbours in $H$. So by~\eqref{Low}, 
\begin{equation}
    \label{d'ismall}
    \text{$d'_i<\mu k$ for $i=1,2$.}
\end{equation}

\begin{claim} 
\label{nobigcomponent2}
If
$T-t$ contains a component $K$ of size exceeding 
$\mu k$ then $d'_1=d'_2=0$. 
\end{claim} 

\begin{proof}
Let $s$ be a neighbour of $t$ in $K$. 
If there is an edge $vy$ with $v\in V(H)$ and $y\in V(G-H-X)$, 
we embed $t$ in $v$ and $s$  in $y$. Applying (\ref{minG-H-X}), we  now embed  $K$ in $G$  using at least 
$\min\{|K|,\frac{k}{3}\}$ vertices of $G-H-X$.  We can then finish the embedding off greedily in $H$.
So, we may assume that $$E(H, G-H-X)=\emptyset,$$ which implies two things: First, by  (\ref{minG-H-X-new}),
\begin{equation}\label{noedgefromHtoG-H-X}
\delta(G-H-X)\ge \frac{k}2-19\mu k,
\end{equation}
 and second,
there is an edge $vy$ with $v\in B_i$ for some $i$ and $y\in V(G-H-X)$, as otherwise $d'_1=d'_2=0$ and we are done. 

Embed $t$ in $v$, and $s$  in $y$. Apply~\eqref{noedgefromHtoG-H-X} to embed  $K$ in $G$   using at least 
$\min\{|K|, k/2-19\mu k\} >\mu k$ vertices of $G-H-X$, and   at most $\max\{0, |K|- k/2+19\mu k\} \le 19\mu k $ vertices of $H$, where we use the fact that $|K|\le k/2$ for the last inequality. As $v\in B_i$, it has degree at least
$k/2+28\mu k$
into $A_{3-i}$, and therefore sees at least $k/2+9\mu k$ unused vertices of $A_{3-i}$.
So by (\ref{DeltaT}) 
we can    embed the 
children of~$t$ in~$A_{3-i}$. We then embed
the remainder of~$T$ greedily in $H$ which has minimum degree at least $(1-\mu)k$, a contradiction to our assumption that $T$ cannot be embedded in $G$.
\end{proof}

We let $H'$ be the bipartite subgraph of $G$ formed by the edges between $A_1 \cup B_1$
and $A_2 \cup B_2$.  
Observation~\ref{minimalityOfG} implies that 
the total number of edges having at least one endpoint in $V(H)\cup X=V(H')\cup (X\cap Y)$  exceeds $\frac{k-2}2 |V(H)\cup X|$. 
Using the definition of $Y$, 
and setting $$e':=e_G(V(H'),V(G-H-X))$$ 
we calculate
 \begin{align}
 e(G[V(H')])\notag  & >  \frac{k-2}2 |V(H)\cup X|  - e(X\cap Y, G-X\cap Y) - e(X\cap Y) -e'\notag \\
 &>  \frac{k-2}2 |H'| +  \frac{(k-2)}2 |X\cap Y|- |X\cap Y|\cdot \big (\frac k2+50\mu k \big ) -e' \notag \\
 &\ge
 \frac{k-2}2 |H'| - 51\mu k |X\cap Y| -e'.\label{pony}
   \end{align}
   
Set
 $$f(v):=
d_G(v,V(H'))
+1000\mu d_G(v,X\cap Y)
+2d_G(v,G-H-X)$$
 for $v \in V(H')$.  Note that by~\eqref{Ai1+3}, \eqref{HighX} and~\eqref{d'ismall}, we have  
\begin{equation}\label{f(v)le22}
  \text{$f(v) \le k+25\mu k$ for all $v\in V(H')$. }
  \end{equation}
Furthermore, using~\eqref{pony} and
 the fact that by the definition of $X$, 
  $$e_G(V(H'),X\cap Y)\ge|E(H,X \cap Y)| \ge \frac{k}{8}|X \cap Y|$$ 
we deduce that the average of $f$ over $v \in V(H')$  exceeds
   $$k-2-\frac{102\mu k|X\cap Y|}{|H'|}-\frac{2e'}{|H'|} +\frac{1000\mu  \frac k8 |X \cap Y|}{|H'|} +  \frac{2e'}{|H'|}
  \ge k-2.
  $$ 
So without loss of generality, and since $|A_1\cup B_1|$ and $|A_2\cup B_2|$ are within a factor of less than two of each other, we may assume that  \begin{equation}\label{avfH'general}
  \text{the average of $f$ over $v \in A_1\cup B_1$  exceeds $f_1:=k-2+x$
  }
  \end{equation}
  \vskip-.3cm
and
\begin{equation}\label{avfH'generalforQ2}
  \text{the average of $f$ over $v \in A_2\cup B_2$  exceeds $f_2:=k-2-2x$
  }
  \end{equation}
 for some  $x\ge 0$.
  In particular, by~\eqref{f(v)le22}, $x< 26\mu k$ and thus for $i=1,2$,
  \begin{equation}
      \label{lun}
      f_i>k-2-52\mu k.
  \end{equation}
  Moreover, setting $$Q_i:=\{v \in A_i \cup B_i ~|~ f(v)>f_i\}$$    for $i=1,2$,  we have that \begin{equation}\label{qnot0}q_i:=|Q_i|>0
      \end{equation}
   and  \begin{equation}\label{eachvofDhas}
  \text{
  each vertex of $Q_i$ has degree exceeding
$f_i-2d'_i$ in $G[V(H)\cup X]$.}
\end{equation}
    Thus,  by~\eqref{HighX} and~\eqref{d'ismall}, since $|A_{3-i}|\le k+3\mu k$ by~\eqref{Ai1+3}, and by the definition of $B_i$,
each vertex of $Q_i\cap B_i$ sees more than
$f_i -2d_i' -19\mu k -\mu k\ge 
|A_{3-i}|-80\mu k$ vertices of $A_{3-i}$ (where we used~\eqref{lun} for the inequality), and so does each vertex of $Q_i\cap A_i$, by our lower bound on $\delta(H)$.
Hence  setting $A_{3-i}':=V_{A_{3-i}}((1-200\mu) q_i, Q_i)$ for $i=1,2$, we have
  \begin{equation}\label{A2toD}
  \text{$|A_{3-i}'|\ge |A_{3-i}|-\frac k2,$ }
  \end{equation}
as otherwise double counting the non-edges between $A_{3-i}$ and $Q_i$ gives $200\mu q_i\frac k2\le 80\mu kq_i$, a contradiction to~\eqref{qnot0}.
  
 For $i=1,2$, set 
 $$Q'_i:=\{v \in A_i \cup  B_i ~|~ f(v) >f_i - 250 \mu q_i\}.$$  
 Note that
\begin{equation}\label{geachvofD'has}
  \text{\hskip-.2cm each  $v\in Q'_i$ has degree exceeding $f_i-250 \mu q_i -2 d'_i$ in $G[V(H)\cup X]$.}
  \end{equation}

Moreover,  we know that 
\begin{equation}\label{Q23}
    |Q'_i|\ge \frac{2k}{3}.
\end{equation}  Indeed, otherwise  by~\eqref{avfH'general} and~\eqref{f(v)le22}, we have $78\mu kq_i \ge 250\mu q_i(|A_i \cup  B_i|-\frac 23k)\ge 250\mu q_i(\frac k3-\mu k)\ge 80\mu kq_i$, a contradiction to~\eqref{qnot0}.

\begin{claim}\label{extend}
  Let $1\le j\le i\le 2$ and $r\in S_i$. Let  $\varphi$ be   an embedding of $r$ and some  
    subset~$\mathcal C$ of the 
    nonsingleton
    components of $T-r$ such that $\varphi (r)\in A_j\cup B_j$ and $\varphi (V(\bigcup\mathcal C))\subseteq V(G-H-X)$. Suppose for $\ell=1+i-j$ we have
   \begin{enumerate}[(a)]
   \item\label{weird}
       $|V(\bigcup\mathcal C)\cap L_2|+f_\ell-2 d'_{\ell}\ge k-2$ and
   \item\label{weird2}
     if $\mathcal C=\emptyset$ then $i=j=1$,  $\varphi(r)\in Q_1$ and $r$ is adjacent to at least two leaves from~$L_2$.
   \end{enumerate}
   Then we can extend $\varphi$ to an embedding of all of~$T$.
\end{claim}
  \begin{proof}
 For convenience, set 
 $D:=V(\bigcup\mathcal C)\cup \{r\}\cup L_2$.
 Embed $T-D$ in a top down fashion in $H'$, starting with the neighbours of $r$, such that
 \begin{itemize}
     \item  $S_2\setminus D$ is embedded in $A_{3-\ell}'$,
     \item if $s\in S_1\setminus D$ is adjacent to $L_2$, we embed $s$ in $Q_\ell$ if possible, and otherwise we embed $s$ in $Q'_\ell\setminus Q_\ell$, and
     \item if $s\in S_1\setminus D$ is not adjacent to $L_2$, we embed $s$   in $Q'_\ell\setminus Q_\ell$ if possible and otherwise we embed $s$ in $Q_\ell$. 
 \end{itemize}
 Such an embedding is possible because $|T-D|\le |T-L_2|\le 2\mu k$ by~\eqref{s1} and~\eqref{T'andL2}, and thus at each step, any vertex from $A_\ell\cup B_\ell$ has an unused neighbour in $A_{3-\ell}'$ by \eqref{A2toD}, and  any vertex  from  $A_{3-\ell}'$ has an unused neighbour in~$Q_\ell'$ by~\eqref{Q23}.
 
Next, if $r\in S_1\setminus\varphi^{-1}( Q_1)$, we  set $L_2^r:=N(r)\cap L_2$ and otherwise we set $L_2^r:=\emptyset$. 
We embed $L_2^r$ in $A_{3-\ell}$. This is possible since $\varphi (r)\in A_\ell\cup B_\ell\subseteq V(G-Y)$ and thus $\varphi (r)$ has degree at least $$\frac k2+50\mu k-d'_\ell-\mu k-19\mu k- |T-D|\ \ge\ \frac k2+27\mu k\ \ge\ \Delta(T)\ \ge\ d(r)$$ into the unused vertices of $A_{3-\ell}$ by~\eqref{T'andL2}, \eqref{DeltaT}, \eqref{HighX}, \eqref{Low} and \eqref{d'ismall}.

 It remains to embed $L_2':=L_2\setminus \big(V(\bigcup\mathcal C)\cup L_2^r)\big)$.
If  all parents of $L'_2$ were embedded  in $Q_\ell$ then
we can embed  $L'_2$ greedily in $G[V(H)\cup X]$  since the vertices of $Q_\ell$ have degree exceeding $f_\ell-2d'_\ell\ge k-2-|V(\bigcup\mathcal C)\cap L_2|$ into $G[V(H)\cup X]$ by~\eqref{eachvofDhas} and~\eqref{weird}, and since $\bigcup\mathcal C$ was embedded in $G-H-X$. So we can assume not all parents of $L'_2$ were embedded  in~$Q_\ell$.

In particular, $|Q_\ell|< k/3$ (as otherwise our second embedding condition and~\eqref{s1} would imply that all parents of $L_2$ are embedded in $Q_\ell$), and thus by~\eqref{Q23}, $|Q'_\ell\setminus Q_\ell|\ge k/3$. So
by our embedding rules,  and since the minimum degree from $A_{3-\ell}$ to $Q'_\ell\setminus Q_\ell$ is at least $|Q'_\ell\setminus Q_\ell|-23\mu k$, each vertex embedded in $Q_\ell$ is adjacent to $L'_2$. Also by our embedding rules, and since not all parents of $L'_2$ were embedded  in~$Q_\ell$,
 we embedded a set $P$ of at least $(1-200 \mu)q_\ell\ge q_\ell/2$ parents of~$L'_2$ in $Q_\ell$, and embedded  $S_1\setminus P$ in   $Q'_\ell$. Note that the set $L_2''$ of  neighbours  of $S_1\setminus P$  in~$L'_2$
has size at most $|L'_2|-q_\ell/2$, a bound which can be improved to $|L'_2|-q_\ell/2-1$ if $\mathcal C=\emptyset$ and thus~\eqref{weird2} applies. 

We  embed $L''_2$ greedily in $G[V(H)\cup X]$. This is possible since by~\eqref{geachvofD'has} and~\eqref{weird}, each vertex of $Q_\ell'$ has degree exceeding 
\begin{align*}
    f_\ell-250\mu q_\ell-2d'_\ell
& \ge k-2-\frac{q_\ell}2-|V(\bigcup\mathcal C)\cap L_2|
\end{align*} into $G[V(H)\cup X]$, which is
sufficient if $\mathcal C=\emptyset$. If $\mathcal C\neq\emptyset$, then   $|V(\bigcup\mathcal C)\cap L_2| \le |V(\bigcup\mathcal C)|-1$  so again the bound is sufficient as $\bigcup\mathcal C$ was embedded in $G-H-X$.

We finish by embedding  $L'_2\setminus L_2''$ greedily in $G[V(H)\cup X]$. This is possible as each vertex of $L'_2\setminus L_2''$ is adjacent to $P$  which was embedded in $Q_\ell$, and since each vertex of $Q_\ell$ has degree  strictly greater than $f_\ell-2d'_\ell\ge k-2-|V(\bigcup\mathcal C)\cap L_2|$ 
in $G[V(H)\cup X]$
   by~\eqref{eachvofDhas} and~\eqref{weird}, while $\bigcup\mathcal C$ was embedded in $G-H-X$.
  \end{proof}

By~\eqref{s1}, some $r\in S_1$ is adjacent to at least two leaves of $T$. 
Thus,~if~$d'_1=0$, we can embed $r$ in $Q_1$ 
and apply Claim~\ref{extend} with $\mathcal C=\emptyset$ and $i=j=1$ to extend this to an embedding of $T$ in $G$, a contradiction. So $d'_1> 0$, and hence,  by Claim~\ref{nobigcomponent2}, 
\begin{equation}\label{d'_inot0} 
    \text{each component of $T-t$ has size at most $\mu k$. }
\end{equation}
\
Now, 
for each $i=1,2$  
let $\mathcal C_i$ be a possibly empty set of  nontrivial components of $T-t$ such that $|\mathcal C_i|\le d'_i$ and $$100 d'_i
    \le 
    |V(\bigcup \mathcal C_i)\cap S_2\cap L|
    \le |V(\bigcup \mathcal C_i)|\le \frac k4.$$
Such  sets exist
      by Lemma~\ref{180p}, 
      which can be applied
because of~\eqref{s1}, \eqref{T'andL2}, \eqref{d'ismall},  \eqref{d'_inot0} 
and
since by 
Observation~\ref{fewleaves}, $t$ is adjacent to at most $\frac k2$ leaves.
\\
If  $t \in S_1$,
we embed~$t$ in~$v_1$ (the vertex of $A_1\cup B_1$ having degree $d'_1$ to $G-H-X$)
and embed $\bigcup\mathcal C_1$ in $G-H-X$, which is possible by~\eqref{minG-H-X}. 
Applying Claim~\ref{extend} with $r=t$ and $i=j=1$,  we can extend this embedding to an embedding of $T$ in $G$, a contradiction. 
Hence we can assume from now on that 
$t$ is in $S_2$.

Note that we could embed~$t$ in~$v_j$ for $j=1$ or $j=2$, 
 use~\eqref{minG-H-X} to embed $\bigcup\mathcal C_j$ in $G-H-X$, and apply Claim~\ref{extend} with $r=t$ and $i=2$ to extend the remainder  of $T$ in $G$, if we can ensure that~\eqref{weird} holds. Hence it would be sufficient to  see that   for at least one of $j=1,2$, we have that $d'_j>0$ and 
 \begin{enumerate}
     \item[{\bf I(j):}] $|V(\bigcup\mathcal C_j)\cap L_2|+f_{3-j}-2 d'_{3-j}\ge k-2$.
 \end{enumerate}
However, if $I(j)$ fails for both $j=1,2$, then by~\eqref{avfH'general} and~\eqref{avfH'generalforQ2},
 \begin{align*}
  3k-6 
  & >\Big(|V(\bigcup\mathcal C_1)\cap L_2|+f_{2}-2 d'_{2}\Big)+2\Big(|V(\bigcup\mathcal C_2)\cap L_2|+f_{1}-2 d'_{1}\Big)
    \\
    & \ge 96 (d'_1+d'_2)+2f_1+f_2
     \\
    & \ge 96 (d'_1+d'_2)+3k-6,
 \end{align*}
 a contradiction. 
 Thus (as we know that $d'_1>0$),  if we are not done, then $I(2)$ holds and $d'_2=0$. In that case $\mathcal C_2=\emptyset$. Spelling out $I(2)$ with $d'_2=0$ and $\mathcal C_2=\emptyset$ gives 
 $f_{1}-2 d'_{1}\ge k-2$. So we embed $r$ in any vertex of $Q_1$ and  apply Claim \ref{extend} with $\mathcal C=\emptyset$ and $i=j=1$ to finish the embedding of $T$. 
 This proves Theorem \ref{thm:main} for subgraphs $H$ as in Theorem \ref{thm:main} (ii).

\subsection{Non-bipartite $H$}

In this section we prove Theorem \ref{thm:main} for subgraphs $H$ as in Theorem \ref{thm:main}~(i). That is, we are given a robust graph $G$ of average degree exceeding $k-2$ and a subgraph~$H$ with $|V(H)|
    \le (1+\mu)k$ and $\delta(H)\ge (1-\mu)k$. For contradiction, we assume the given tree $T$ cannot be embedded in $G$.
    
Set $H_0:=H$ and for $i=1,\ldots ,2\mu k$, let $H_i$ be obtained from $H_{i-1}$ by adding a vertex $v_i$ of $G-H_{i-1}$ which  has at least $\frac{(3+100\mu)}{4}k$ neighbours in $H_{i-1}$ if such a vertex exists, and by setting $H_i:=H_{i-1}$ otherwise. Set $H':=H_{2\mu k}$.

Let $a_i$ be the average degree of $H_i$. Then for $1\le i\le 2\mu k$, either $H_i=H_{i-1}$ or 
\begin{align}
a_i\ & \ge \ a_{i-1}\frac{|H_{i-1}|}{|H_i|}+ 2\cdot \frac{\frac{3+100\mu}4 k}{|H_i|}\ \ge \
\frac{a_{i-1}(|H_{i}|-1)+\frac 32 k+50\mu k}{|H_{i}|}  \notag  \\
&
\ge \ a_{i-1}+
\frac{-a_{i-1}+\frac 32 k+50\mu k}{|H_{i}|}
\ \ge \
a_{i-1}+
\frac{\frac 12 k+47\mu k}{k+3\mu k}    > \ a_{i-1}+\frac 12, \label{increase}
\end{align}
   where we used the facts that $H_i$ has at most $(1+3\mu)k$ vertices and that $a_{i-1}\le |H_{i-1}|$ for the second-to-last inequality.

Let $a$ denote the average degree of $H'$. If  $ |H'-H|\ge 2\mu k$ then by \eqref{increase}, 
$$
a=a_{2\mu k} \ge a_0 + |H'-H|\cdot \frac 12 \ge (1-\mu)k +  2\mu k\cdot \frac 12 =k,
$$
and we can apply Theorem~\ref{thmRS25} to embed $T$ and are done. 
   Thus we may assume  that 
   \begin{equation}\label{not2muk}
   |H'-H|<2\mu k,
   \end{equation}
   and therefore
\begin{equation}\label{outsideH'baddeg}
   \text{each $v\in V(G-H')$ has less than $\frac{3+100\mu}{4}k$ neighbours in $H'$.}
   \end{equation}
   
   Moreover, if $H'$ has a vertex $v^*$ of degree at least $k-1$, then we are done by Theorem~\ref{thmRS} applied to the graph $\tilde H$ consisting of $v^*$ and $k-1$ of its neighbours, since the minimum degree of $\tilde H$ is at least 
   
   $$
  \delta(H')-|H'-\tilde H| \ge 
  \frac{(3+100\mu)}{4}k
  - 3\mu k >\frac 23 k
   $$
   where we used \eqref{not2muk}, the given bounds on the minimum degree and the order of $H$, and the definition of $H'$. So we  will assume from now on that
   \begin{equation}\label{mindegH'casenonbip}
   \text{each $v\in V(H')$ has degree at most $k-2$ in $H'$.}
   \end{equation}
  Set $$\bar a:=k-2-a.$$ By \eqref{increase}, and by the given bound on the minimum degree of $H$, we know that
    \begin{equation}\label{bara}
  \bar a <\mu k.
   \end{equation}
   We claim that furthermore, 
     \begin{equation}\label{bara2}
  \bar a >0.
   \end{equation}
   Indeed assume otherwise. Then $H'$ is $(k-2)$-regular, and by the robustness of $G$, there is an edge from $H'$ to a vertex $w\in V(G-H')$. We can embed $T$ greedily top-down into $H'\cup\{w\}$, by rooting $T$ at any leaf which we embed in $w$. So~\eqref{bara2}~holds.

 Next, 
we define a set $X$ as follows. Starting with the empty set, we  repeatedly add vertices from $V(G-H')$ which have less than $11 \mu k $ neighbours in  $G-H'-X$ until either $|X|\ge  3 \bar a$  or  no suitable vertices remain in~$G-H-X$. We note that by Observation~\ref{minDegofG}, the average degree  from $X$ to $H'$ is greater than $(k-1)/2-11\mu k -3 \bar a\ge 49|H'|/100$, where we use~\eqref{bara} for the inequality. Hence
\begin{equation}\label{baranew}
    \text{
the average degree from $H'$ to $X$ exceeds $\frac{49}{100}|X|$.
    }
\end{equation}
and therefore, the average degree from $H'$ to $H'\cup X$ exceeds $a+\frac{49}{100}|X|\ge k-2-\bar a+ \frac{|X|}3$. So,
\begin{equation}\label{baranew2}
    \text{if $|X|\ge  3 \bar a$ then  $H'$ has a vertex of degree at least $k-1$ in $G$.}
\end{equation}
Furthermore, setting $G':=G-H'-X$, we have
     \begin{equation}\label{XtoG'}
\text{$|E(X,G')|< 11 \mu k|X|$}, 
   \end{equation}
  and   
   \begin{equation}\label{mindeggg}
\text{if $|X|<3\bar a$ then $\delta(G')\ge 11 \mu k$}.
   \end{equation}
Let $t$ be the $k/2$-separator of $T$.     
 \begin{claim}\label{thasatmostmukleafch}
       \text{Vertex 
$t$ has at most  $ \mu k$ leaf children.}
          \end{claim} 
        \begin{proof}
        Assume otherwise. Then the set $L$ of leaves of $T$ that are adjacent to~$t$ has size exceeding $\mu k$. We claim that   
        \begin{equation}\label{Xinsideclaim}
            |X|< 3\bar a.
        \end{equation} Indeed, otherwise~\eqref{baranew2} implies that $H'$ has a vertex $w$ of   degree at least $k-1$ in $G$ (and degree at least $k/2$ in $H$). We embed $t$ in $w$, and embed the remainder of $T-L$   greedily in~$H$. 
    This is possible because of the high minimum degree of $H$ and  since $t$ has at most $\frac{k}{2}$ non-leaf children    while $w$ has more than $\frac{k}{2}$ neighbours in $H$. We finish by embedding $L$ greedily in $G$. This proves~\eqref{Xinsideclaim}.
    
    Let $v^*$ be  
         a vertex from $V(H')\cup X$ that maximises $$f(v):=|N(v) \cap (V(H')\cup X)|+2|N(v)\cap V(G')|.$$ Note that  
            \begin{equation}\label{fv*}
        f(v^*)>k-2
           \end{equation}
         as otherwise, summing $f(v)$ over all $v\in V(H')\cup X$, we obtain that 
          $$
          |V(H')\cup X|\cdot (k-2) \ge 2e(V(H')\cup X) + 2e(V(H')\cup X,V(G')),
          $$
          a contradiction to Observation~\ref{minimalityOfG}. 
         By~\eqref{outsideH'baddeg} and using~\eqref{bara} and~\eqref{Xinsideclaim}, each vertex $x\in X$ has less than  $\frac{(3+100\mu)k}{4}+3\bar a \le \frac{3k}{4}+ 30\mu k$ neighbours in $V(H')\cup X$. Furthermore, the definition of $X$ implies that $x$ has less than $11\mu k$ neighbours in $G'$. So $v^*\notin X$ and thus by~\eqref{fv*},
         $v^*\in V(H')$. Hence by~\eqref{not2muk} and by the definition of $H'$, it follows that 
         \begin{equation}\label{v*morethan}
     \text{$v^*$ has  more than $\frac{3+100\mu}{4} k-2\mu k\ge \frac{3k}{4}+20\mu k$ 
neighbours in~$H$.}
\end{equation}

{\bf Case 1: $|N(v^*)\cap V(G')|\ge \mu k$.} We  embed in $G'$ a set of  $\mu k$ leaves adjacent to~$t$. 
We  embed all remaining neighbours of $t$ in $H$. This is possible by~\eqref{v*morethan} and since by Observation~\ref{fewleaves} we can assume $t$ has at most  $3k/4$ neighbours. 
We finish by embedding the remainder of $T$ in $H$ (which has minimum degree at least $(1-\mu)k$). 

{\bf Case 2:  $|N(v^*)\cap V(G')|< \mu k$.}  It follows from~\eqref{not2muk}, \eqref{bara}, \eqref{fv*} and~\eqref{Xinsideclaim} that  
 \begin{equation}\label{nbsofv*inHin2ndcase}
 \text{$|N(v^*)\cap V(H)|\ge k-2-2\mu k-|X|-|H'-H| 
 >k-11\mu k$.}
 \end{equation}

 Since $t$ is a $k/2$-separator, each nonsingleton component of $T-t$ has size at most $k/2$, while there are at most $k/2\le k-2d-1$ leaves adjacent to~$t$ by Observation~\ref{fewleaves}. So because of Remark~\ref{theremark}, we can apply Lemma~\ref{subsetofcompsofT2}   with input $b:=k/2$ and $d:=|N(v^*)\cap V(G')|$   to obtain a set $\mathcal C$ of at most $d$ nonsingleton components of $T-t$ containing between $2d$ and $\frac{k}{2}$ vertices.
 
We embed the components of $\mathcal C$, starting with their roots, which are embedded in $G'$, and embedding the remainder of $\bigcup\mathcal C$ anywhere  in~$G$, using as much of $G'$ as possible.
 We can do this as $\delta(G)> (k-1)/2$ by Observation~\ref{minDegofG}. By~\eqref{mindeggg} and~\eqref{Xinsideclaim},   at least $\min\{11\mu k, 2d\}$ vertices of $\bigcup\mathcal C$ are embedded in~$G'$, and if we embed any  of $\bigcup\mathcal C$
 in $H' \cup X$ then at least $11\mu k$ vertices are embedded in $G'$.
 So by~\eqref{nbsofv*inHin2ndcase}, we can embed all remaining neighbours of $t$ in $H$, which we do if $11\mu k\le 2d$. In this case we set $L':=\emptyset$. If $2d<11\mu k$ we let $L'$ be a set  of $\mu k$ leaves adjacent to $t$, and
  embed in $H$ all other  remaining neighbours of $t$ (while leaving $L'$ for a later stage).

We then embed the  remainder of $T-L'$ greedily  in $H$, which is possible as $\delta(H)\ge (1-\mu)k$, and since either $11\mu k$ vertices were embedded in~$G'$ or $|L'|=\mu k$. We are done if  $L'=\emptyset$. Otherwise,
 we finish by embedding $L'$  into unused neighbours of~$v^*$. This is possible since $v^*$ has more than $k-2-2d$ neighbours in  $V(H')\cup X$ by~\eqref{fv*},
 and since  at least $2d$ vertices were embedded in $G'$.   
          \end{proof}

     Set     
          $$H'':=G[X\cup V(H')]$$ and note that 
          
\begin{equation}\label{H'''}
|H''|\le |H'|+|X|  
\le k+6\mu k+1
\end{equation}
 by \eqref{not2muk}  and by the definition of $X$ and~\eqref{bara}. 
Set $$H^*:=G[V(H)\cup V_{H'}(k-1, H'')].$$
Then $H\subseteq H^*\subseteq H'\subseteq H''.$
By~\eqref{not2muk} and by  the definition of $X$ and~\eqref{bara},
        \begin{equation}\label{H*mindeg}
       \delta(H^*)\ge 
       \min \big\{(1-\mu)k, k-1-|(H'-H)\cup X| \big\} 
       \ge  k-6\mu k.
        \end{equation}
        and

         \begin{equation}\label{H*size}
       |H^*|\le  |H'|\le |H|+2\mu k\le k+3\mu k.
        \end{equation}
        In particular,
        \begin{equation}\label{H*nonadajcencies}
       \text{each  $v\in V(H^*)$ sees all but at most
$9\mu k$  vertices of $H^*$.}
  \end{equation}
  
We need the following claim.

\begin{claim}
\label{Aug14claim1} If $T$ has at least $100\mu k$ leaves then $|V_{H'}(k-1,H'')| < 33\mu k$.
\end{claim}

\begin{proof}
    Suppose the claim is not true and let $Z^*$ be an arbitrary set of $33\mu k$ vertices of $V_{H'}(k-1, H'')$.
       Let $L$ be the set of all leaves of  $T$.   By~\eqref{H*mindeg}, we can embed   $T-L$ in $H^*$ obeying the following  conditions: 
       
  \begin{enumerate}[(a)]
  \item parents of leaves of $T$ are embedded in $Z^*$ if possible; and
  \item all other vertices are embedded in $H^*-Z^*$ if possible.
  \end{enumerate}
  Note that $|H^*-Z^*|\ge k-34\mu k$, and so by~\eqref{H*mindeg} and since $|L|\ge 100\mu k$, we can always obey condition~(b). Also, by~\eqref{H*nonadajcencies}, we either embedded all parents of $L$ in $Z^*$ or we used all but at most $9\mu k$ vertices of $Z^*$ for parents of $L$. Either way, for  a set $L'\subseteq L$ of size at least $20\mu k$ it holds that  their parent is embedded in $Z^*$. 
  We embed $L\setminus L'$ in $H^*$ which is possible by~\eqref{H*mindeg}. We finish by embedding $L'$ greedily. This yields an embedding of $T$, a contradiction.
\end{proof}

The next claim has a similar outcome as Claim~\ref{Aug14claim1}, but the opposite premise.

\begin{claim}
\label{Aug14claim2}
 If $T$ has at most  $100\mu k$ leaves then $|V_{H'}(k-1,H'')|< \frac{k}{9}$.
\end{claim}

\begin{proof}
Suppose the claim fails and 
     take an arbitrary subset $V^*$ of $V_{H'}(k-1, H'')$ with $|V^*|=\frac k{10}$. 
We let $t$ be a $\frac{99}{100}k$-separator of $T$. We let $T'$ be a tree 
of size between $\frac{98}{100}k$ and $\frac{99}{100}k$ consisting of $t$ and the union of some of the 
components of $T-t$.

We embed $t$ in an arbitrary vertex of $V^*$ and embed as much as possible of $T'$  greedily top-down in $H''$. While doing so, we obey the following rules:
\begin{enumerate}[(a)]
\item if possible, we embed in $H''-H^*$;
\item if we cannot embed in $H''-H^*$, then if possible we embed in $H^*-V^*$;
\item when choosing which  vertex from $T'$ to embed next (among those with embedded parents), we give priority to vertices whose parent is embedded in $H''-H^*$.
\end{enumerate}


Consider  a vertex $s\in V(T')$ which  was embedded that has a child $s'$ in $T'$ which we could not embed. Then by~\eqref{H*mindeg}, $s$ was embedded in $V(H''-H^*)$, and by rule~(c),
 at each step of our embedding after embedding $s$, we have embedded a vertex whose parent was embedded in $H''-H^*$.  
 As 
 $|H''-H^*|\le |H'-H|+|X| \le 5\mu k+1 \le 12\mu k$ by~\eqref{not2muk} and~\eqref{bara}, there are at most $12\mu k$ such parents.  Since $T$ has at most  $100\mu k$ leaves, any fixed subset $P$ of $T$ has at most $|P|+100\mu k$ children, and thus, we only continued embedding for at most $112\mu k$ more steps after embedding~$s$.

On the other hand, in the step before embedding $s$ into some $v\in V(H''-H^*)$, we know that by rule~(a), any vertex from~$N_{H''}(v)$ that has already  been used for the embedding either accommodates a leaf of~$T$, one of the at most $12\mu k$ parents of  vertices embedded in $H''-H^*$,   or~$t$. Thus, there are at most $112\mu k+1$ used vertices in the neighbourhood of $v$ at the moment we embed $s$ in $v$. Hence, at the end of our embedding procedure, there are at most $112\mu k+1+112\mu k=225\mu k$ used vertices in the neighbourhood of $v$, which means we could  have embedded $s'$, a contradiction.

We just proved that
\begin{equation}\label{allofT}
\text{we embedded all of $T'$ following rules (a)--(c).}
\end{equation}

We claim that furthermore
\begin{equation}\label{Zfree}
\text{we used all of $H''-H^*$ in the first $\frac{3}{5}k$ steps of our
embedding.}
\end{equation}
For contradiction, suppose $v\in V(H''-H^*)$ was not used after $\frac{3}{5} k$ steps. As $v\in H''$ and by Observation~\ref{minDegofG}, we know that $v$ sees at least $\frac{k-1}2-11\mu k$ vertices of $H''$. Thus $v$ sees at least $\frac{k-1}2-11\mu k-\frac 25k\ge\frac k{11}$ used vertices of $H''$, of which at most $113\mu k$ are  images of  leaves,  parents of vertices that were embedded in $H''-H^*$, or $t$. 
So $v$ sees at least $\frac k{11}-201\mu k>0$ vertices of $H''$ which accommodate 
 parents of vertices embedded in $H^*$, which means we violated rule~(a) by not using $v$. So~\eqref{Zfree} holds.
 
 Moreover, since we embedded $T$ in a top-down fashion and 
 as $T$ has at most $100\mu k$ leaves, there are at most  $100\mu k$ children of vertices embedded in $H''-H^*$, and these are embedded before other vertices, by rule (c).   So, by~\eqref{Zfree}, for each $w\in V(T)$, 
\begin{align}\label{Zfree2}
\hskip-.1cm\text{
if $w$ is not embedded 
by step $\frac{7}{10} k$, then $w$'s parent is embedded in $H^*$.}
\end{align}
 We claim that 
\begin{equation}\label{Zfree3}
\text{we used all of $H''-V^*$ in the first $\frac{19}{20} k$ steps of our
embedding.}
\end{equation}

Assuming~\eqref{Zfree3} is true, 
 we are able to embed $T-T'$ in $V^*$. Indeed, this is possible as $t$ was embedded in $V^*$ and as by definition, the vertices of $V^*$ all  have degree at least $k-1$ in $H''$ and all their neighbours in $H''-V^*$ have already been used.

All that remains is to prove~\eqref{Zfree3}. 
For this, note that apart from $t$ and  the at most $100\mu k$ children of vertices embedded in $H''-H^*$, we did not embed any vertices of $T'$ into $V^*$ until $H^*-V^*$ had less than $10\mu k$ unused vertices, by~\eqref{H*nonadajcencies} and by rule (b). Let $j$ be the first step  at which the set $S$ of unused vertices of
$H^*-V^*$  has size below $10 \mu k$. 
Then at the end of step $j$ we have used at most $101 \mu k$ vertices of $V^*$ 
and hence $j \le  |H^*-V^*|-10\mu k +101\mu k$.

At step $\frac{19}{20} k$, consider any vertex $v\in S$. Note that $\frac{19}{20} k>j+200\mu k$ by~\eqref{H*size}. Of the last $200\mu k$ vertices that were embedded, at most $100\mu k $ are leaves by hypothesis. By~\eqref{H*nonadajcencies}, by~\eqref{Zfree2} and since $S\subseteq V( H^*)$, at most $10\mu k$ are not adjacent to $v$ and since $|S|<10\mu k$,
 less than $10\mu k$ are the parent of a vertex embedded in $S-v$. 
 So at least $80\mu k>0$ of these vertices had children which could have been embedded in $v$, and thus $v$ must have been used. This proves~\eqref{Zfree3}. 
\end{proof}

We are now able to show that $X$ is small.

\begin{claim}
\label{claimwas(60)}
  $|X| < 2.8\bar a$.
\end{claim}

 \begin{proof}
 By Claims \ref{Aug14claim1} and \ref{Aug14claim2}, we know that
$|V_{H'}(k-1, H'')| <\frac{k}{9}$.   Hence at most~$\frac k9$ vertices of $H'$ have degree exceeding $k-2$,
while also having degree at most $(k-2)+|X|$. 
Moreover, by~\eqref{baranew}, the average degree from $H'$ to $X$ exceeds $\frac{49}{100}|X|$. 
Thus $$
\frac{48}{100}k |X| \le |H'|\cdot \frac{49}{100} |X| \le |E(H',X)| < \frac k9 \cdot |X| + |H'|\cdot \bar a,
$$
 and therefore, by~\eqref{not2muk},
 $
\frac{36}{100}k |X| <  |H'|  \bar a \le (1+3\mu ) \bar a k.
$
Hence $|X| < 2.8\bar a$, as desired.
 \end{proof}

Combining Claim~\ref{claimwas(60)} with~\eqref{mindeggg} we obtain
\begin{equation}
    \label{mindegggforreal}
    \delta(G')\ge 11 \mu k.
\end{equation}
 
 Let $b$ be the average degree  of the vertices of $H'$ in~$H''$,
 that is, we set $$b:=\frac{1}{|H'|}\sum_{v\in V(H')}d(v, H'').$$

 By~\eqref{mindegH'casenonbip}, the maximum degree of $H'$ is bounded from above by $k-2$.
       So 
$$|H'|\cdot (k-2)+|V_{H'}(k-1, H'')|\cdot |X|\ge b|H'|,$$
and thus, by Claim \ref{claimwas(60)}, we obtain that
\begin{equation}\label{VH'H'''}
\text{$2.8\bar a \cdot |V_{H'}(k-1, H'')| \ge (b-k+2)|H'|$.}
\end{equation}

                 \begin{claim}\label{avdegH'toH'cupX}
If $T$ has at least $100\mu k$ leaves then  $b\le k-2+100\mu \bar a.$
          \end{claim}
          
        \begin{proof}
        Suppose $T$ has at least $100\mu k$ leaves but  $b> k-2+100\mu \bar a.$ 
       Then by~\eqref{VH'H'''}, $$|V_{H'}(k-1, H'')|\ge \frac{100\mu \bar a}{2.8\bar a}|H'|> 33\mu k,$$
       a contradiction to Claim \ref{Aug14claim1}.
    \end{proof}

   In the proof of the next claim, the following quick definitions and observations will be useful. (We will use the set $Z'$ in the next claim, while the set~$Z$ is used only in the proof of Claim~\ref{Tatmost200leaves}, however, because of the similarity of these sets, it seems easiest to introduce them both here.)
 Set $$Z:=V_{H'}(k-2-\frac{101}{100}\bar a, H')
\text{ \ \ and \ \ }
Z':=V_{H'}(k-2-100\bar a, H').
 $$
By \eqref{mindegH'casenonbip}, we know that $\sum_{v\in Z}d_{H'}(v)\le |Z| \cdot (k-2)= |Z| \cdot (a+\bar a)$, while $$\sum_{v\in V(H')\setminus Z}d_{H'}(v)< |V(H')-Z|(k-2-\frac{101}{100}\bar a) \ \le \ |V(H')-Z| (a-\frac{\bar a}{100}).$$
So as $a$ is the average degree of $H'$, we have $ |Z| \cdot \bar a\ge |H'-Z|\cdot \frac{1}{100}\bar a $ and thus 
    \begin{equation}\label{ZZZZ}
    |Z| \ge \frac{|H'|}{101} \ge \frac{(1-\mu) k}{101} \ge \frac{k}{102}.
   \end{equation}
  Similarly, $ |Z'|\cdot \bar a \ge  |H'-Z'|\cdot 99\bar a$ and thus
    \begin{equation}\label{ZZZZ'}
    |Z'| \ge  \frac{99|H'| }{100}\ge   \frac{89}{90}k.
   \end{equation}

 Note that  $$H''':=G[V(H) \cup Z']$$ 
 contains at most $|H'|\le k+3 \mu k$
vertices and has minimum degree at least $\min\{(1-\mu)k, k-2-100\bar a\}\ge k-2\mu k-100\bar a$.
So
    \begin{equation}\label{ZZZZadj}
 \text{each vertex  in  $H'''$ sees all but at most $5\mu k +100\bar a$  vertices of $H'''$.}
          \end{equation}
                
                Let $\Delta'$ be the maximum over the degrees to $G'$ of vertices from $H'$.
            
               \begin{claim}\label{notcase2.1'}
     If $T$ contains a set $L$ of  $100\mu k$ leaves 
     then   there is no subset $\mathcal C$ of the nontrivial components of
     $T-t$ with $|\mathcal C|\le \Delta'$ and $200\bar a\le |\bigcup \mathcal C|\le k/2$.
         \end{claim}

  \begin{proof}       
        Assume otherwise. 
        Let $v$ be a vertex in $H'$ with   $\Delta'$ neighbours in~$G'$. We embed $t$ in $v$. 
Observation~\ref{minDegofG} allows us to embed  $\bigcup\mathcal C$, starting with the roots, which are embedded in $G'$, and embedding the remainder of $\bigcup\mathcal C$ anywhere  in~$G$, using as much of $G'$ as possible. 
          Let $x$ be the number of vertices  of $\bigcup\mathcal C$ embedded in $G'$. By~\eqref{mindegggforreal}, we know that $x\ge \min\{200\bar a, 11\mu k\}$.

Next, we map the remaining  neighbours of $t$ 
to $H$. 
This is possible 
because~$v\in V(H')$ has  at least $(3+100\mu)k/4-2\mu k$ neighbours in~$H$, while by Claim~\ref{thasatmostmukleafch}, the number of used vertices of $H$ plus the number of neighbours of $t$ outside $\bigcup \mathcal C$ is at most $$|\bigcup \mathcal C|+ \frac{|V(T)-t-\bigcup \mathcal C|}2 + 
         \frac{\mu k}2\le \frac{k-1}2+\frac{|\bigcup \mathcal C|}2+ \frac{\mu k}2<\frac {3+90\mu}4k.$$
         If no vertices of $H$ were used for $\bigcup\mathcal C$, we can (and will) choose to embed all nonleaf neighbours of $t$ in  $Z'$, by~\eqref{ZZZZ'}.
        
 Now, if $x> \mu k$, 
then we can embed the remainder of $T$ greedily in $H$.       
         So we can assume that $x\le\mu k$,  and thus, by~\eqref{mindegggforreal}, only vertices from $G'$ have been used for the embedding so far, and the  neighbours of $t$ were embedded in  $Z'$.
    
    We embed the remaining vertices of $T-L$ top-down from these vertices  in~$H'''$ as follows. If $w$ has a neighbour in $L$ then we  wish to embed $w$ into $Z'$, and otherwise, we  wish to embed $w$ into $H- Z'$. By~\eqref{ZZZZadj}, we can embed as we wish until we have either used all but $5\mu k+100\bar a$ vertices of $Z'$ or all but $5\mu k+100\bar a$ vertices of $H-Z'$. 
     In the former case we would have used $|Z'|-5\mu k-100\bar a$ vertices of $Z'$ for parents of $L$ and for $N(t)$, implying the contradiction $$100\mu k=|L|\ge |Z'|-5\mu k-100\bar a - |N(t)|\ \ge \ \frac k3,$$  where we used~\eqref{bara} and \eqref{ZZZZ'}   for the second inequality, as well as the fact that  $t$ has at most $\mu k$ leaf neighbours by Claim~\ref{thasatmostmukleafch}, and at most $k/2$ non-leaf neighbours.
     
    So we are in 
     the latter case, i.e.~we used all but $5\mu k+100\bar a$ vertices of $H-Z'$. We continue our embedding, now trying to embed in $Z'$ all remaining vertices of $T-L$. We are only forced to stop when each of $Z'$ and $H-Z'$ have at most  $5\mu k+100\bar a$ unused vertices. However, we embedded at least $200\bar a$ vertices outside $H'$, and have not embedded any    vertices of the $100\mu k$ vertices of  $L$ yet. Hence, as $|H'''|\ge |H|\ge (1-\mu )k$, we are never forced to stop, and can embed the rest of $T-L$ in $H'''$.
     It remains to embed~$L$ which we can do since all neighbours of $L$ were embedded in $Z'$, and because of the definition of $Z'$ and the fact that we  embedded at least $200\bar a$ vertices of $T$ in~$G'$.
     \end{proof}

               \begin{claim}\label{notcase2.1}
     If $T$ contains a set $L$ of  $100\mu k$ leaves then  $\Delta'< 100\bar a$. 
         \end{claim}
                 \begin{proof} 
                  Assume that $\Delta'\ge 100\bar a$. 
We let $\mathcal C$ be as given by Lemma~\ref{subsetofcompsofT2} 
 for input 
$b:=k/2$ and $d:=\max\{1,\lceil100\bar a\rceil\}$. We can use this lemma because of Claim~\ref{thasatmostmukleafch} and Remark~\ref{theremark}. We reached a contradiction to Claim~\ref{notcase2.1'}.
             \end{proof}

                   Next, let us bound the number of edges in $H'''$ and use this information to bound the number of edges from  $H'''$ to $G'$. These considerations will be used in the next claim, and also at the end of this section.
                   
      By~\eqref{outsideH'baddeg}, 
         there are less than $\frac{3+100\mu}4 k|X|$ edges between $H'$ and $X$. So, letting $$q:=b-a$$
         denote the average degree from $H'$ to $X$, we can write $$|E(H',X)| \le   \frac 23 \cdot\frac{3+100\mu}4 k|X| +  \frac 13\cdot q |H'| \le \frac{k-2}2|X| + 17\mu k |X| + \frac q3 |H'|.$$
              Moreover,   $|X|\le 3\mu k$ 
              by Claim~\ref{claimwas(60)} 
         and~\eqref{bara} and thus $X$ contains at most $\frac{|X|^2}2\le \frac 32 \mu k|X|$ edges. Finally, $H'$ has  $\frac{a}2|H'|=\frac{k-2-\bar a}2|H'|$ edges. Putting all of this together, we see
         $H''=G[V(H') \cup X]$   contains
         \begin{align*}
         |E(H',X)|  +    |E(H')|+|E(X)| 
        &  \le
          \frac{k-2}2|X|  + \frac q3 |H'| + \frac{k-2}2|H'|-\frac{\bar a}2|H'| + 19\mu k |X|
         \\ &
            \le
       \frac{k-2}2 | V(H')\cup X| -  \big(\frac{\bar a}2- \frac q3 \big) |H'| +19 \mu k|X|
         \end{align*}
          edges. Hence, by Observation~\ref{minimalityOfG}, there are more than  $(\frac{\bar a}2 -\frac q3) |H'| -19 \mu k|X|$ edges 
 leaving   $H''$, i.e.~going from $H''$ to $G'$. 
  So by~\eqref{XtoG'} and  using~\eqref{not2muk}, \eqref{bara} and Claim~\ref{claimwas(60)}
  for the second inequality,  we see that 
   \begin{align}\label{EHG'2}
  |E(H, G')| &> (\frac{\bar a}2 -\frac q3) |H'| -30 \mu k|X|-\Delta' |H'-H| \notag \\ & \ge (\frac{\bar a}2 -\frac q3) |H'| -100 \mu \bar a  k   -2\mu \Delta' k. \end{align}  
  
 We are ready for the next claim.
        
              \begin{claim}\label{Tatmost200leaves}
           \text{$T$ has fewer than $100\mu k$ leaves.}
          \end{claim}     
          \begin{proof}
          Assume $T$ contains a set $L$ of  $100\mu k$ leaves. Then by
          Claim~\ref{avdegH'toH'cupX},  
          \begin{equation}\label{aboutb}
            \text{$b\le k-2+100 \mu \bar a$,}
         \end{equation}
         and by
         Claim~\ref{notcase2.1}, we have $\Delta'< 100\bar a$, that is,
              \begin{equation}\label{aboutH'100}
        \text{each vertex  of $H'$ has less than $100\bar a$ neighbours in $G'$.}
         \end{equation}
         Moreover, by~\eqref{EHG'2}
    and by~\eqref{aboutH'100} (applied to the vertices of  $H'-H$),  we have   \begin{equation}\label{EHG'}
  |E(H, G')| >  \big(\frac{\bar a}2 -\frac q3\big) |H'| -300 \mu \bar a  k.\end{equation} 
  By~\eqref{aboutb}, 
 $q =b-a\le k-2+100\mu \bar a  -  (k-2-\bar a)= (1+100\mu)\bar a$ and so \eqref{EHG'} becomes 
$$
|E(H, G')| > \big(\frac{\bar a}6 -100\mu\bar a\big)\cdot |H'|  -300 \mu \bar a  k> \frac{\bar a}{10}\cdot |H|.
$$ 
In particular, $\Delta'>\bar a/{10}$. 
Claim~\ref{thasatmostmukleafch} and Remark~\ref{theremark} allow us to use Lemma~\ref{subsetofcompsofT2} with input $b:=k/2$ and $d:=\max\{1,\lceil \bar a/{10}\rceil\}$ to obtain a set of at most $\lceil d\rceil$ components of $T-t$ containing between $2d$ and $k/2$ vertices. 
We let $\mathcal C$ be as given by Lemma~\ref{subsetofcompsofT2} for this input, and among all choices for such a $\mathcal C$ we choose one with $|\bigcup\mathcal C|$ maximum.     Claim~\ref{notcase2.1'},    implies  that $|\bigcup\mathcal C|<200\bar a$.
         Hence, letting $\mathcal C'$ denote the set of all nontrivial components of $T-t$ having at most $2000$ vertices, and noting that by our choice of $\mathcal C$, any nontrivial component of $T-t-\mathcal C$ lies in $\mathcal C'$, we deduce that 
 $|\bigcup\mathcal C'|>k-1-200\bar a-\mu k\ge \frac{999}{1000}k$, where the term $\mu k$ represents the singleton components whose number is bounded by Claim~\ref{thasatmostmukleafch}, and we used~\eqref{bara} for the inequality. So 
     \begin{equation}\label{thesmallcomps}
     |\mathcal C'|>\frac k{4000}.
       \end{equation}

Next, we claim that there is a vertex $w\in V(H)$ and a matching 
$M$ from $H \cap N(w)$ to $G-V(H')-(N(w) \cap V(G'))$ such that 
    \begin{equation}\label{thematching}
        \text{$|M|+2|N(w)\cap V(G')| \ge (1+\frac{1}{20})\bar a$.}
         \end{equation}
         
         Before we show there are $w$ and $M$ satisfying~\eqref{thematching}, let us prove that we can embed $T$ using such $w$ and $M$.
         We embed $t$ in $w$. We set 
$$x:=\min\big\{|M|+|N(w)\cap V(G')|, 2\mu k\big\}$$ and let $\mathcal C_x$ be a subset of $\mathcal C'$ of size $x$ which exists by~\eqref{thesmallcomps}.
We embed 
 the roots of  the components  from $\mathcal C_x$ in $N(w) \cap (V(G')\cup V(M))$. The remainder of these components will be embedded at a later step.
 
By~\eqref{thesmallcomps}, $\mathcal C'\setminus\mathcal C_x$ has a subset $C''$ of size   $\mu k$. Delete one leaf from each  component in $\mathcal C''$ and call the resulting set of components~$\mathcal C_L$ and the set of deleted leaves $L$. 
We embed $\bigcup\mathcal C_L$   in the set $Z\cap V(H)$ which by~\eqref{ZZZZ}, \eqref{not2muk}, \eqref{bara} and by the definition of $Z$, has size at least $\frac k{103}$ and minimum degree at least $\frac{k}{104}$. 
Since  $|\bigcup\mathcal C_L|\le 2000\mu k$, and $w$ sees all but at most $2\mu k$ vertices of $Z\cap V(H)$, such an embedding is possible.

Next, we embed the remainder of the components  from $\mathcal C_x$ as follows. First, for each root embedded in $V(M)$ we embed one of its neighbours in the other endpoint of the edge from $M$. Then, we embed the remainder of  $\bigcup\mathcal C_x$ greedily, preferring vertices in $G'$. 
Observe that by~\eqref{bara}, \eqref{mindegggforreal} and~\eqref{thematching}, after this step
we embedded at least 
$(1+\frac{1}{20} )\bar a$
vertices of $T$ outside  $H'$.
We continue by 
 greedily embedding all that remains of $T-L$ in $H$, which is possible since we have not embedded the set $L$ of size $\mu k$ yet and since $\delta(H)$ is large enough. Finally we greedily embed  $L$, which we can do   since the parents of leaves from $L$ were embedded  in vertices of~$Z$, which by definition have degree at least $\lceil k-2-\frac{101}{100}\bar a\rceil \ge  k-1-\lceil(1+\frac{1}{20})\bar a\rceil$ in $H'$.

It only remains to see that there are $w$ and $M$ fulfilling~\eqref{thematching}. For this, we consider a maximum matching~$M_1$ from $H$ to $G'$. The size of $M_1$ equals the size of a minimum cover for the
edges  between $H$ and $G'$. By~\eqref{aboutH'100}  and by~\eqref{outsideH'baddeg}, each vertex in $H\cup V(G')$ sees less than  $\frac{3+100\mu}{4}k\le \frac{31}{40}k$ of these edges. So by~\eqref{EHG'}, it follows that 
$$|M_1|\ge \frac{|E(H,G')|}{ \frac{31}{40}k}\ \ge \
\Big( \frac{20}{31}\bar a  - \frac{40}{93} q \Big)\frac{|H'|}k- 400 \mu \bar a 
\ \ge\ \frac{5}{8} \bar a-\frac 49 q.$$

By the definition of $q$ as the average degree from $H'$ to $X$, and since the average degree from $X$ to $H'$ is below $ \frac 45|H'|$ by~\eqref{outsideH'baddeg}, we know that $|X|\ge \frac 54q$. Moreover, each vertex of $X$ has degree at least $$\frac{k-1}2-|H'-H|-|X| - 11\mu k \ge \frac k3$$ into~$H$, where we use  Observation~\ref{minDegofG},  \eqref{not2muk}, \eqref{bara}, Claim~\ref{claimwas(60)}  and the definition of $X$.
So we can greedily add some edges from $X$ to $H$ in order to  extend $M_1$ to a matching $M_2$ from $H$ to $G-H'$ with $$|M_2|\ \ge \ \frac 58 \bar a-\frac 49 q+\frac 54 q\ \ge \ \frac 58 \bar a+\frac 45 q.$$

Observe that since each vertex of $V(M_2)\cap V(H)$ sees all but at most $3\mu k$ of the vertices of $H$, it follows that 
on average, a vertex of $H$ sees at least $(\frac 58 \bar a+\frac 45 q)(1-4\mu)\ge  \frac 59 \bar a+ \frac 23 q$ vertices of  $V(M_2)\cap V(H)$. 
Moreover, by~\eqref{EHG'} the average 
degree of a vertex from $H$ into $G'$ is at least 
 $(\frac{\bar a}2 -\frac q3) -400 \mu\bar a$. 
Therefore  the average 
degree of a vertex from $H$ into $ ( V(H) \cap V(M_2))\cup V(G')$   is at least  
$$ \frac 59 \bar a+ \frac 23 q + \frac{\bar a}2 -\frac q3 -800 \mu\bar a \ \ge\ ( 1 + \frac 1{20})\bar a.$$
In particular there is a vertex $w\in V(H)$ with $|N(w)\cap \big ((V(H)\cap V(M_2))\cup V(G') \big)|\ge (1 + \frac 1{20})\bar a$. We let $M\subseteq M_2$ consist of all edges of $M_2$ that meet $N(w)\cap V(H)$ and do not meet $N(w)\cap V(G')$. Then  each  edge in $M_2\setminus M$ that meets $N(w)\cap V(H)$ also meets
  a vertex of $N(w)\cap V(G')$, and thus,
$$|M|+2|N(w)\cap V(G')|\ge 
|N(w)\cap V(H)\cap V(M_2)|+|N(w)\cap V(G')|\ge
(1 + \frac 1{20})\bar a$$ 
Hence there are $w$ and~$M$ satisfying~\eqref{thematching}.
We have thus proved Claim~\ref{Tatmost200leaves}.
                      \end{proof}

                     Observe that if $b> k-2+\frac{\bar a}{3}$, then
 by~\eqref{VH'H'''},
 $|V_{H'}(k-1, H'')|\ge \frac{|H'|}{8.4}> \frac{k}9$
 which is impossible  by Claims \ref{Aug14claim2} and \ref{Tatmost200leaves} and since $|H'|\ge(1-\mu)k$. 
Thus $b\le k-2+\frac{\bar a}{3}$,
and therefore, 
 $$q =b-a\le k-2+\frac{\bar a}{3}  -  (k-2-\bar a)= \frac 4{3}\bar a.$$ 
So by~\eqref{EHG'2}, 
$|E(H, G')| \ > \ \frac{\bar a}{18}|H'| -100 \mu \bar a  k   -2\mu \Delta' k\ \ge \ \frac{\bar a}{19}|H|    -2\mu \Delta' k.$
On the other hand, $|E(H, G')| \le \Delta'|H|$. Thus 
$$ \frac{\bar a}{19}|H|  <    |E(H, G')| +2\mu \Delta' k 
\le 
\Delta' (|H| + 2\mu   k)<2 \Delta' |H| ,$$
that is
\begin{equation}\label{bara38}
\Delta'  >   \frac{\bar a}{38}.
\end{equation}
Let  $v^*\in V(H')$  be a vertex with  at least $\frac{\bar a}{38}$ neighbours in $G'$.
Applying Claim \ref{Tatmost200leaves} twice, and noting that each component of $T-t$ has a leaf,  we obtain first that the  sum of the sizes of the nonsingleton components of $T-t$ is at least $(1-100 \mu)k$  and second that their average size is at least $\frac{(1-100\mu)k}{100\mu k}>10^6$. 
So,
 by~\eqref{bara38}, if there are at least $\Delta'$ nonsingleton components of $T-t$, the $\Delta'$  largest  components together have at least $10^6\Delta'>10^6\frac{\bar a}{38} \ge 200\bar a$ vertices. Otherwise, this holds for the set of all nonsingleton components.
 Hence, 
letting $\mathcal C$ be a minimal subset of these components  of $T-t$ with the property that $|\bigcup\mathcal C| \ge 200\bar a$, we have  

\begin{equation}\label{100baraa}
200\bar a\le  |\bigcup \mathcal C|\le k/2,
\end{equation}
where the second inequality holds since $t$ is a $\frac k2$-separator and since $\bar a  <\mu k$ by~\eqref{bara}. 
 Let $\mathcal C'$ be the set of all 
 other components of $T-t$.

We embed $t$ in $v^*$.  Observation~\ref{minDegofG} allows us to embed  $\bigcup\mathcal C$, starting with the roots, which are embedded in $G'$, and embedding the remainder of $\bigcup\mathcal C$ anywhere  in~$G$, using as much of $G'$ as possible. Let $x$ be the number of vertices  of $\bigcup\mathcal C$ embedded in $G'$. By~\eqref{mindegggforreal}, we know that 
\begin{equation}\label{ifx11}
   \text{if $ x <11\mu k$ then $x=|\bigcup \mathcal C|$} 
\end{equation}
and that at this point
 \begin{equation}\label{aboutv*}
         \text{$v^*$ sees at least $\frac 34k -(|\bigcup \mathcal C|-x)$ unused vertices of~$H$.}
         \end{equation}

  If $x> \mu k$ then we can embed the remainder of $T$ greedily in $H$. For this,  note that the remaining neighbours of $t$  can be mapped to $H$ 
because the quantity from~\eqref{aboutv*}  is at least $k/4$ 
          and  because 
          by Claim~\ref{Tatmost200leaves}, 
           $t$ has at most $100\mu k$ neighbours in total. 
         Hence by~\eqref{ifx11}, we may assume from now on that      \begin{equation}\label{aboutxlemuk}200\bar a \le |\bigcup \mathcal C| =  x\le \mu k.     \end{equation}
          where we use~\eqref{100baraa} for the first inequality.
Thus, the only vertex outside $G'$ used in the embedding so far is $v^*$.

By~\eqref{ZZZZadj}, each vertex of $H'''=G[V(H) \cup Z']$ sees all but at most $5\mu k +100\bar a$  vertices of $H'''$. By~\eqref{not2muk}, $H'-H'''$ has at most $2\mu k$ vertices, each seeing all but at most $k/4$ vertices of $H$, and thus all but at most $k/3$ vertices of $Z'$, by~\eqref{ZZZZ'}.
 Hence  there is  a set $\mathcal P$ of $|H'-H'''-\{v^*\}|$  disjoint paths $p_1p_2p_3$ with $p_2\in V(H'-H''')$ and $p_1, p_3\in Z'$, and such that all $p_i$ are unused.
Let $P_i$ be the set of all vertices $p_i$, for $i=1,2,3$.

We embed the set $S$ of roots of the components of $\mathcal C'$ in  $Z'$, avoiding $P_1\cup P_3$, which is possible by~\eqref{aboutv*} and~\eqref{aboutxlemuk}, and since $t$ has at most $100\mu k$ children, that is $|S|\le 100\mu k$. 
We choose 
a set $\mathcal Q$ of $|H-Z'-v^*|$  disjoint paths $q_1q_2q_3$ with $q_2\in V(H-Z')$ and $q_1, q_3\in Z'$,  such that all $q_i$ are unused and distinct from $P_1\cup P_3$.
This is possible as $$|H-Z'|\le (1+\mu)k+|H'-H|-\frac{89}{90}k\le \frac{k}{50}$$ by~\eqref{not2muk} and~\eqref{ZZZZ'}, and since we have used at most $100\mu k$ vertices of $H'$ so far, while $|P_1\cup P_3|\le 4\mu k$ and $\delta(H)\ge (1-\mu )k$.

Now, $\bigcup\mathcal C'-S$  has  at least $k-1-x-|S|\ge k/2$ vertices and contains at most 
$100 \mu k$ leaves of $T$. A standard argument shows that thus, $\bigcup\mathcal C'-S$ contains a set $\mathcal R$ of $|\mathcal P|+|\mathcal Q|+2\mu k$ disjoint 5-vertex paths, all of whose internal vertices have degree 
2. We order the components of $\mathcal C'$ arbitrarily, and  embed each top-down from its root in~$Z'$,   except that, 
as long as there is an unused path  $P\in {\mathcal P} \cup {\mathcal Q}$ whenever we embed  an endvertex $x_1$ of a path $x_1x_2x_3x_4x_5$ from $\mathcal R$, we immediately embed $x_3x_4x_5$ in $P$. We then embed $x_2$ in a common neighbour of the images of $x_1$ and $x_3$. We go on from $x_5$ as usual. We avoid using   vertices on   paths of ${\mathcal P} \cup { \mathcal Q}$ otherwise. We can  continue until we have used all  paths of  ${\mathcal P} \cup { \mathcal Q}$, by \eqref{ZZZZadj} and the fact that the  number of unembedded vertices of $T$ in the paths of $\mathcal R$ exceeds the sum of the sizes of the unused elements of $\mathcal{P} \cup {\mathcal Q}$ by at least  $5(|\mathcal R|-|\mathcal P|-|\mathcal Q|)= 10\mu k > 5\mu k+ 100\bar a$, where we used~\eqref{aboutxlemuk} for the inequality.   Once all paths from  $\mathcal P\cup \mathcal Q$ have been used, we embed the remainder of $T$ in $Z'$ which is possible since by definition, each vertex of $Z'$ has degree at least $k-2-100\bar a$ in $H'$, all vertices of $H'-Z'$ have been used by now, and more than $200\bar a$ vertices of $T$ were embedded in~$G'$ by~\eqref{aboutxlemuk}.
This finishes  
 the proof.

\newcommand{\etalchar}[1]{$^{#1}$}

\end{document}